\documentclass[11pt,reqno]{amsart}
\usepackage{amsmath,amssymb,amsthm,amscd, amsxtra, amsfonts,graphicx,fancybox,cancel,framed,longtable,booktabs,bm}
\usepackage[symbol]{footmisc}
\usepackage{wasysym}

\usepackage[all]{xy}
\usepackage{amssymb}
\usepackage{enumitem}
\usepackage{latexsym}
\usepackage{amscd}
\usepackage{color}
\usepackage{tikz}
\usepackage{verbatim}
\usepackage[draft]{hyperref}

\newtheorem{theorem}{Theorem}[section]
\newtheorem{lemma}[theorem]{Lemma}
\newtheorem{definition}[theorem]{Definition}

\newtheorem{proposition}[theorem]{Proposition}

\newtheorem{conjecture}[theorem]{Conjecture}
\newtheorem*{remark}{Remark}

\newcommand{\R}{\mathbb R}

\newcommand{\bea}{\begin{eqnarray}}
\newcommand{\eea}{\end{eqnarray}}
\newcommand{\be}{\begin {equation}}
\newcommand{\ee}{\end{equation}}

\newcommand{\Z}{\mathbb Z}

\newcommand{\N}{\mathbb N}
\newcommand{\C}{\mathbb C}
\newcommand{\Q}{\Bbb Q}

\newcommand{\IN}{\mathbb N}
\setlist[enumerate]{leftmargin=*, listparindent=\parindent, parsep=0pt,
	font=\upshape, label=\alph*)} 
\setlist[itemize]{leftmargin=*} 

\renewcommand{\b}[1]{{\boldsymbol{#1}}}

\newcommand{\sgn}{{\rm sgn}}
\renewcommand{\pmod}[1]{\  \,  \left(  \mathrm{mod} \,  #1 \right)}

\usepackage{amssymb}

\allowdisplaybreaks

\numberwithin{equation}{section}

\makeatletter
\newcommand{\vast}{\bBigg@{4}}
\newcommand{\Vast}{\bBigg@{5}}
\makeatother
\makeatletter
\def\smallunderbrace#1{\mathop{\vtop{\m@th\ialign{##\crcr
				$\hfil\displaystyle{#1}\hfil$\crcr
				\noalign{\kern3\p@\nointerlineskip}%
				 \tiny\upbracefill\crcr\noalign{\kern3\p@}}}}\limits}
\makeatother

\begin{document}

\title{Higher depth quantum modular forms and plumbed $3$-manifolds}

\author{Kathrin Bringmann, Karl Mahlburg, Antun Milas}

\address{University of Cologne, Department of Mathematics and Computer Science, Weyertal 86-90, 50931 Cologne, Germany}
\email{kbringma@math.uni-koeln.de}

\address{Department of Mathematics, Louisiana State University, Baton Rouge, LA 70803, USA}
\email{mahlburg@math.lsu.edu}

\address{Max Planck Institute for Mathematics, Vivatsgasse 7, 53111 Bonn, Germany }
\address{Permanent address: Department of Mathematics and Statistics, SUNY-Albany, Albany, NY 12222, U.S.A.}
\email{amilas@albany.edu}

\thanks{The research of the first author is supported by the Alfried Krupp Prize for Young University Teachers of the Krupp foundation and the research leading to these results receives funding from the European Research Council under the European Union's Seventh Framework Programme (FP/2007-2013) / ERC Grant agreement n. 335220 - AQSER. The third author was supported by NSF-DMS grant 1601070 and a stipend from the Max Planck Institute for Mathematics, Bonn.}



\begin{abstract}
In this paper we study new invariants $\widehat{Z}_{\b a}(q)$ attached to plumbed $3$-manifolds that were introduced by Gukov, Pei, Putrov, and Vafa. These remarkable $q$-series at radial limits conjecturally compute WRT invariants of the corresponding plumbed $3$-manifold. Here we investigate the series $\widehat{Z}_{0}(q)$ for unimodular plumbing {\tt H}-graphs with six vertices.
We prove that for every positive definite unimodular plumbing matrix, $\widehat{Z}_{0}(q)$ is a depth two quantum modular form on $\mathbb{Q}$.
\end{abstract}

\maketitle

\section{Introduction and statement of results}

A {\it quantum modular form} is a complex-valued
function defined on $\mathbb{Q}$ or subset thereof, called the {\em quantum set}, that
exhibits modular-like transformation properties up to an obstruction term with ``nice'' analytic properties (for instance, it can be extended to a real-analytic function on some open subset of $\mathbb{R}$).
Quantum modular forms were introduced by Zagier in \cite{ZagierQuantum}, where he described several non-trivial examples. They have appeared in several areas including quantum invariants of knots and 3-manifolds \cite{Hikami0,Hikami2,Hikami,LZ}, mock modular forms \cite{Zw},  meromorphic Jacobi forms \cite{BRZ}, mathematical physics \cite {DMZ}, partial and false theta functions \cite{BM}, and representation theory \cite{BM,CMW}.

Motivated on the one hand by the concept of higher depth mock modular forms and on the other hand by the appearance of higher rank false theta functions in representation theory, Kaszian and two of the authors  \cite{BKM} defined  so-called higher depth quantum modular forms, and gave an infinite family of examples coming from characters of representations of vertex algebras.
If the depth is two, these functions satisfy
\begin{equation*}
f(\tau) - (c \tau+d)^{-k} f(\gamma \tau) \in \mathcal{Q}^1  \mathcal{O}(R),\qquad \gamma=\left(\begin{smallmatrix}
a & b \\ c & d
\end{smallmatrix}\right)\in{\rm SL}_2(\Z),
\end{equation*}
where $\mathcal{Q}^1$ is the space of quantum modular forms and $\mathcal{O}(R)$ is the space of real-analytic functions on $R$.
All known examples of depth two  quantum modular come from rank two partial theta functions ($q:=e^{2\pi i\tau}, \tau\in\mathbb{H}$)
\begin{equation*} \sum_{{{\bf n} \in  \mathbb{N}^2_{0}+{\b{\beta}} }}  q^{an_1^2+b n_2^2+c n_1 n_2},\end{equation*}
where $\b{\beta} \in \mathbb{Q}^2$ (throughout we write vectors in bold letters and their components with subscripts) and  $a, 4ab-c^2 >0$. Further examples of this kind were studied in \cite{BKM2,Males}. Depth two quantum modular forms also appear as the coefficients of meromorphic Jacobi forms of negative matrix index \cite{BKMZ}.

In \cite{GPPV}, as a part of the construction of homological invariants for closed 3-manifolds, Gukov, Pei, Putrov, and Vafa proposed a new approach to  WRT invariants for a large class of $3$-manifolds. For any  plumbed $3$-manifold, homeomorphically represented by a plumbing graph and positive definite linking matrix $M$  \footnote{In \cite{GPPV}, $M$ is negative definite, which we account for by replacing it with $-M$ when referring to their work.}, they \cite{Ne} defined a certain family of  $q$-series (called {\it homological blocks})
\begin{equation}
\label{Gukov:ZqDef}
\widehat{Z}_{\b a}(q):= \frac{q^{\frac{-3N+{\rm tr}(M)}{4}}}{(2\pi i)^N} \text{PV} \int_{|w_j|=1} \prod_{j=1}^N g(w_j) \prod_{(k,\ell) \in E} f(w_k,w_\ell) \Theta_{-M, \b a}(q;{\b w}) \frac{dw_j}{w_j},
\end{equation}
where PV denotes the Cauchy principle value, where throughout integrals are oriented counterclockwise and $\int_{|w_j|=1}$ indicates the integration $\int_{|w_1|=1}\ldots\int_{|w_N|=1}$. Moreover  $g(w_j)$ and  $f(w_k,w_\ell)$ are certain simple rational functions defined in \eqref{f} and \eqref{g}, respectively and
\begin{equation*}
\Theta_{-M, \b a}(q;{\b w}):=\sum_{\boldsymbol{\ell} \in 2 M \mathbb{Z}^N+\b a } q^{\frac14\boldsymbol{\ell} ^T M^{-1} \boldsymbol{\ell} } {\b w}^{{\b \ell}}, \ \ \b{a} \in 2{\rm coker}(M)+ \b \delta,
\end{equation*}
where $\b \delta:=(\delta_j)$ such that $\delta_j \equiv {\rm deg}(v_j) \pmod 2$ with $\delta_j$ denoting the {\it degree} (or valency) of $j$-th node.
Conjecturally, a suitable (explicit) linear combination of  $\widehat{Z}_{\b a}(q)$, denoted by $\widehat{Z}(q)$ in  \cite{GPPV}, is the universal WRT invariant, that is,  as $q \to e^{\frac{2 \pi i }{k}}$ its limit
coincides  with the SU$(2)$ WRT invariant of $M$ at level $k$. This, in particular,  leads to another conjecture (attributed in \cite{BMM} to Gukov) that $\widehat{Z}_{\b a}(q)$ and $\widehat{Z}(q)$ are quantum modular
forms. This conjecture can be verified for specific $3$-manifolds obtained from unimodular 3-star plumbing graphs
(e.g. the $E_8$ graph) \cite{BMM, CCFGH}  due to the fact that $\widehat{Z}_{\b a}(q)$ can be expressed via one-dimensional unary false theta functions
\begin{equation*}
\sum_{n \in \mathbb{Z}} {\rm sgn}(n) q^{an^2+bn},
\end{equation*}
whose quantum modularity properties are well-understood \cite{BM, Hikami2,Hikami,LZ,Zw}.

In this paper we investigate $\widehat{Z}_{\b a}(q)$ for a family of  non-Seifert plumbed $3$-manifolds. We  consider
the simplest  plumbing graph of this kind obtained by splicing two $3$-star graphs. This way we obtain the so-called ${\tt H}$-{graph} with six vertices (Figure \ref{H}), with the linking matrix
\begin{center}
\begin{figure}
			\begin{tikzpicture}
			\node[shape=circle,fill=black, scale = 0.4] (1) at (0,0) { };
			\node[shape=circle,fill=black, scale = 0.4] (2) at (-1,1) { };
			\node[shape=circle,fill=black, scale = 0.4] (3) at (-1,-1)  { };
			\node[shape=circle,fill=black, scale = 0.4] (4) at (1.5,0) { };
			\node[shape=circle,fill=black, scale = 0.4] (5) at (2.5,-1) { };
			\node[shape=circle,fill=black, scale = 0.4] (6) at (2.5,1) { };

						\node[draw=none] (B1) at (0,0.4) {$b_3$};
			\node[draw=none] (B2) at (-0.6,-1) {$b_2$};
			\node[draw=none] (B3) at (-0.6,1) {$b_1$};
			\node[draw=none] (B4) at (1.5, 0.4) {$b_4$};
			\node[draw=none] (B5) at (2.1,-1) {$b_6$};
			\node[draw=none] (B6) at (2.1,1) {$b_5$};

			\path [-] (1) edge node[left] {} (2);
			\path [-](1) edge node[left] {} (3);
			\path [-](1) edge node[left] {} (4);
			\path [-](4) edge node[left] {} (5);
			\path [-](4) edge node[left] {} (6);
		
			\end{tikzpicture}
			\caption{The {\tt H}-graph} \label{H}
			\end{figure}
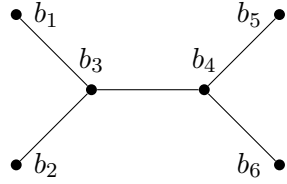
		\end{center}
\begin{equation} \label{b-matrix}
M=\left(
\begin{smallmatrix}
b_1 & 0 & -1 & 0 & 0 & 0 \\
0 & b_2 & -1 & 0 & 0 & 0 \\
-1 & -1 & b_3 & -1 & 0 & 0 \\
0 & 0 & -1 & b_4 & -1 & -1 \\
0 & 0 & 0 & -1 & b_5 & 0 \\
0 & 0 & 0 & -1 & 0 & b_6 \\
\end{smallmatrix}
\right).
\end{equation}
We only consider positive definite unimodular matrices whose  $3$-manifolds are integral homology spheres (i.e., $H_1(M_3,\mathbb{Z})=0$ as explained further in Section \ref{sec:Prelim:Def} below). Due to the invariance of $\widehat{Z}_{{\b \delta}}(q)$ under a Kirby move \cite{GPPV}, we may assume that $b_j \geq 2$, $j \in \{1,2,5,6\}$ (graphs with $b_j=1$, $j \in \{1,2,5,6\}$ reduce to $3$-star graphs whose  quantum modularity is well-understood \cite{BMM,CCFGH}).
With these assumptions $\widehat{Z}_{\b \delta}(q)$ (also denoted by $\widehat{Z}_{0}(q)$ in \cite{GPPV})  is the only homological block and therefore it conjecturally gives
WRT invariants at roots of unity. An important feature of this family of graphs is that $\widehat{Z}_{\b \delta}(q)$ can be expressed via rank two false theta functions ($\b\beta\in\Q^2$, $a,b,c\in\N$)
\begin{equation*}
\sum_{{\bf n} \in  \mathbb{Z}^2 }  {\rm sgn}^*(n_1){\rm sgn}^*(n_2)  q^{a(n_1+\beta_1)^2+b (n_2+\beta_2)^2+c (n_1+\beta_1)(n_2+\beta_2)},
\end{equation*}
where ${\rm sgn}^*(x):={\rm sgn}(x)$ for $x \in \R \setminus \{0\}$ and ${\rm sgn}^*(0):=1$. Our first result is on quantum modularity of certain partial theta functions needed to study $\widehat{Z}_{\b \delta}(q)$.

More generally, we prove quantum modularity of an infinite family of false theta functions which we now introduce. Define
\begin{equation*}
F_{\mathcal S,Q,\varepsilon}(\tau) :=\sum_{\b\alpha\in\mathcal S}\varepsilon(\b \alpha) \sum_{\b n \in\N_0^2} q^{KQ(\b n+\b \alpha)},
\end{equation*}
where $\mathcal S\subset \Q^2 \cap (0,1)^2$ is a finite set with the property that  $(1,1)-\b{\alpha}$, $(1-\alpha_1,\alpha_2)$, $(\alpha_1,1-\alpha_2)\in\mathcal S$ for $\b{\alpha}\in\mathcal S$, $\varepsilon:\mathcal S\to \C$ satisfies $\varepsilon(\b{\alpha})=\varepsilon((1,1)-\b{\alpha})=\varepsilon((1-\alpha_1,\alpha_2))$, and $K\in\N$ is minimal such that $K\mathcal S\subset \N^2$. For convenience,
we extend the domain of $\varepsilon$ to $\mathcal{S} + \mathbb{Z}^2$ by letting $\varepsilon(\b \alpha)= \varepsilon(\b \alpha+\b n)$, $\b n \in \mathbb{Z}^2$.
\begin{theorem}\label{QMT-intro}
	The function $F_{\mathcal S,Q,\varepsilon}$ is a quantum modular form of depth two, weight one, and quantum set $\mathcal Q_{\mathcal S,Q,\varepsilon}$, defined in
	(\ref{quantum-FS}).
\end{theorem}
\noindent Theorem \ref{QMT-intro} is of independent interest and can be used to investigate other examples of quantum modular forms.

Next we move on to studying unimodular matrices arising from  ${\tt H}$-graphs. Since the graph has six vertices it is not  surprising that there are only finitely many positive definite unimodular matrices. We prove the following result.
\begin{theorem} \label{39} There are, up to graph isomorphism, precisely $39$ equivalence classes of unimodular positive definite plumbing matrices  (\ref{b-matrix}) with $b_j \geq 2$, $j \in \{1,2,5,6 \}$.
\end{theorem}
Then our main result is the following.
\begin{theorem}
\label{T:mainquantum}
For any positive definite unimodular plumbing matrix as in Theorem \ref{39} , $q^{c_M} \widehat{Z}_{0}(q)$, for some $c_M \in \mathbb{Q}$, is a quantum modular form of depth two, weight one, and quantum set $\mathbb{Q}$.
\end{theorem}
Based on our results here and in \cite{BMM}, we can  slightly reformulate Gukov's conjecture mentioned in \cite{BMM} on the quantum modularity of $\widehat{Z}_{\b a}(q)$ and $\widehat{Z}(q)$.
\begin{conjecture}\label{conj:plumb}
Let $T$ be a plumbing graph (tree)  with $r$ nodes of degree at least three. Then $\widehat{Z}_{\bold{a}}(q)$ is a depth $ r$ quantum modular form whose quantum set is a subset of $\mathbb Q$. Moreover, for any unimodular plumbing matrix, $\widehat{Z}(q)$ is quantum of depth $r$
with quantum set $\mathbb{Q}$.
\end{conjecture}
Combined with the conjecture on $\widehat{Z}(q)$ mentioned above, Conjecture \ref{conj:plumb} would imply that (unified) WRT invariants of plumbed $3$-manifolds are higher depth quantum modular forms. We expect that the higher depth property also holds true for higher rank ${\rm SU}(N)$ invariants (see \cite{C}).

The paper is organized as follows. In Section 2, we discuss special functions, the Euler-Maclaurin summation formula, higher depth quantum modular forms, and double Eichler integrals.  In Section 3
we show quantum modularity of $F_{\mathcal S,Q_1,\varepsilon}$ (see Theorem \ref{QMT}).
In Section 4, we prove our main result on quantum modularity of ${Z}(q)$, defined in (\ref{E:ZqDef}), for unimodular plumbing graphs (see Theorem \ref{main-thm}).
The proof of the classification of positive definite unimodular matrices (\ref{b-matrix}) is given in Section 5. Finally, in the appendix we list data for all $39$ equivalence classes of positive unimodular matrices needed to
compute ${Z}(q)$.

\smallskip

{\bf Acknowledgements:} The  authors  thank S. Chun, S. Gukov,  and C. Manolescu for helpful discussion on some aspects of \cite{GPPV} .

\section{Preliminaries}
\label{sec:Prelim}

\subsection{Special functions}

Following \cite{ABMP} (with slightly different notation),
for each $\kappa\in \R$ we define a function $E_2:\R\times\R^2\rightarrow\R$ by
\begin{equation*}
E_2(\kappa;\boldsymbol{x}):=\int_{\R^2}\sgn\left(w_1\right)\sgn\left(w_2+\kappa w_1\right)e^{-\pi \left(\left(w_1-x_1\right)^2+\left(w_2-x_2\right)^2\right)}dw_1dw_2.
\end{equation*}
For $x_2,x_1-\kappa x_2\neq0$, we set
\begin{align*}
M_2(\kappa; \b x):=-\tfrac{1}{\pi^2} \int_{ \R^2-i \b x} \frac{e^{-\pi w_1^2-\pi w_2^2-2\pi i(x_1w_1+x_2w_2)}} {w_2(w_1-\kappa w_2)} dw_1dw_2.
\end{align*}
The following formula relates $M_2$ and $E_2$
\begin{multline}\label{splitM}
	M_2(\kappa;x_1+\kappa x_2,x_2) = E_2(\kappa; x_1+\kappa x_2,x_2) + \sgn(x_1)\sgn(x_2)  \\
	- \sgn(x_2) E(x_1+\kappa x_2)-\sgn(x_1) E\left(\tfrac{\kappa }{\sqrt{1+\kappa^2}}x_1 + \sqrt{1+\kappa^2}x_2\right),
\end{multline}
where for $x\in\R$, we set $E(x):=2\int_{0}^{x}e^{-\pi w^2}dw$.

The proof of the next result follows from the proof of \cite[Lemma 6.1]{BKM}. Here $\tau=u+ i v$.
\begin{proposition}For $\kappa, x_1, x_2 \in \R$ we have
\begin{align}\label{Mrep}
&M_2(\kappa;x_1,x_2)=
-\frac{x_1}{2\sqrt{v}}
\frac{ x_2}{\sqrt{v}}
q^{\frac{x_1^2}{4v}+\frac{x_2^2}{4v}}
\int_{-\overline{\tau}}^{i \infty}
\frac{e^{\frac{\pi i  x_1^2 w_1}{2v}} }{\sqrt{-i(w_1+\tau)}}
\int_{w_1}^{i\infty}
\frac{e^{\frac{\pi i x_2^2 w_2}{2v}}}{\sqrt{-i(w_2+\tau)}}dw_2
dw_1
\\ &\quad -
\frac{x_2+\kappa x_1}{2\sqrt{(1+\kappa^2)v}}
\frac{x_1-\kappa x_2}{\sqrt{(1+\kappa^2)v}}
q^{\frac{(x_2+\kappa x_1)^2}{4\left(1+\kappa^2\right)v}
	+\frac{(x_1 - \kappa x_2)^2}{4\left(1+\kappa^2\right)v}}
\int_{-\overline{\tau}}^{i \infty}
\frac{e^{\frac{\pi i(x_2+\kappa x_1)^2w_1}{2\left(1+\kappa^2\right)v}} }{\sqrt{-i(w_1+\tau)}}
\int_{w_1}^{i\infty}
\frac{e^{\frac{\pi i(x_1-\kappa x_2)^2w_2}{2\left(1+\kappa^2\right)v} }}{\sqrt{-i(w_2+\tau)}}
dw_2
dw_1.
\notag
\end{align}
\end{proposition}

\subsection{Euler-Maclaurin summation formula}

Let $B_m(x) $ be the $m$-th Bernoulli polynomial  defined by $\frac{w e^{xw}}{e^w - 1} =: \sum_{m \geq 0} B_m(x) \frac{w^m}{m!}$. We require
\begin{equation} \label{Ber}
B_m(1-x) = (-1)^m B_m(x).
\end{equation}

The Euler-Maclaurin summation formula implies the following lemma.
\begin{lemma}\label{EulerMcLaurin}
For $\boldsymbol{\alpha}\in\R^2$, $F:\R^2\rightarrow\R$ a $C^\infty$-function which has rapid decay, we have
\begin{align*}
\notag
&\sum_{\boldsymbol{n}\in\N_0^2}  F((\boldsymbol{n}+\boldsymbol{\alpha})t) \\
\notag
&\sim \frac{\mathcal I_F}{t^2} - \sum_{n_2 \geq 0} \frac{B_{n_2+1}(\alpha_2)}{(n_2+1)!} \int_{0}^{\infty}F^{(0,n_2)}(x_1,0)dx_1t^{n_2-1}- \sum_{n_1 \geq 0} \frac{B_{n_1+1}(\alpha_1)}{(n_1+1)!}\int_{0}^{\infty}F^{(n_1,0)}(0,x_2)dx_2t^{n_1-1}  \\
&\quad+ \sum_{n_1,n_2\geq 0}\frac{B_{n_1+1}(\alpha_1)}{(n_1+1)!}\frac{B_{n_2+1}(\alpha_2)}{(n_2+1)!} F^{(n_1,n_2)}(0,0)t^{n_1+n_2},
\end{align*}
where $\mathcal I_F:=\int_{0}^{\infty}\int_{0}^{\infty} \allowbreak F(\boldsymbol{x})dx_1dx_2$.
Here by $\sim$ we mean that the difference between the left- and the right-hand side is $O(t^N)$ for any $N\in \N$.
\end{lemma}

\subsection{Gauss sums}
We define for $a,b,c\in\Z$ with $c>0$ the {\it quadratic Gauss sums} \begin{equation*}G_c(a,b):= \sum_{n \pmod c} e^{\tfrac{2 \pi i}{c} \left( a n^2 +b n \right)};\end{equation*}
see \cite[Section 1.5]{Br} for some basic properties. We use the following elementary result on the vanishing of $G_c(a,b)$.

\begin{proposition}
\label{P:G=0}
If $\gcd(a,c)\nmid b$, then $G_c(a,b)=0$.
\end{proposition}

\subsection{Shimura theta function}
We require certain theta functions studied, for example, by Shimura \cite{Sh}.
For $\nu\in\{0,1\}$, $h\in\Z$, $N,A\in\N$, with $A|N$, $N|hA$, define
\begin{equation*}
\vartheta_\nu (A,h,N;\tau):=\sum_{\substack{m\in\Z \\ m\equiv h\pmod{N}}}m^\nu  q^{\frac{Am^2}{2N^2}}.
\end{equation*}
Define the {\it slash operator} of weight $k\in\frac12\Z$ ($(\frac{\,\cdot\,}{\,\cdot\,})$ the Jacobi symbol)
\begin{equation*}
f\big|_k\gamma (\tau):= \left(\tfrac{c}{d}\right)^{2k} \varepsilon_d^{2k}(c\tau+d)^{-k}f(\gamma\tau),\qquad \gamma=\left(\begin{smallmatrix}
a & b \\ c & d
\end{smallmatrix}\right)\in{\rm SL}_2(\Z).
\end{equation*}
Note that if $k\in\Z+\frac12$, we require that $\gamma\in\Gamma_0(4)$. Recall that Shimura's modular transformation formula \cite[Proposition 2.1]{Sh} states that for $\gamma=\left(\begin{smallmatrix} a&b\\c&d \end{smallmatrix}\right) \in \Gamma_0(2N)$, with $2|b$, we have
\begin{equation}\label{Shimura1}
\vartheta_\nu (A,h,N;\tau)\mid_{\tfrac 32} \gamma = e\left(\tfrac{abAh^2}{2N^2}\right)\left(\tfrac{- 2A }{d}\right)\vartheta_\nu(A,ah,N;\tau).
\end{equation}
Here $e(x):=e^{2\pi ix}$, for odd $d$, $\varepsilon_d=1$ or $i$, depending on whether $d\equiv 1\pmod{4}$ or $d\equiv 3\pmod{4}$.

\subsection{Integral evaluations}
\label{sec:Prelim:Int}

We require, for $m\in\Z$,
\begin{align}\label{int1}
\int_{|w|=1}\left(w-w^{-1}\right)w^m\frac{dw}{w}=\int_{|w|=1}w^mdw- \int_{|w|=1}w^{m-2}dw =2\pi i\left(\delta_{m,-1}-\delta_{m,1}\right),
\end{align}
where $\delta_{m,a}=0$ unless $m=a$ in which case it equals $1$ and
\begin{equation}\label{int2}
\frac{1}{2\pi i}\text{PV}\int_{|w|=1}\frac{w^m}{w-w^{-1}}\frac{dw}{w} = \frac12 {\rm sgn}_o(m),
\end{equation}
where ${\rm sgn}_o(m):=\frac{1}{2}{\rm sgn}(m)(1-(-1)^m)$.

\subsection{Higher depth quantum modular forms}
We now give the formal definition of quantum modular forms, following \cite{ZagierQuantum}.

\begin{definition}
	A function $f:\mathcal Q\to \C$ ($\mathcal{Q} \subseteq \mathbb{Q}$) is called a {\it quantum modular form of weight $k \in \frac12\Z$ for a subgroup $\Gamma$ of ${\rm SL}_2(\Z)$ (of $\Gamma_0(4)$ if $k\in\Z+\frac12$) and quantum set $\mathcal Q$} if for $\gamma=\left(\begin{smallmatrix}
	a & b\\ c & d
	\end{smallmatrix}\right)\in\Gamma$, the function
	\begin{equation*}
	f(\tau)-f\big|_k\gamma (\tau)
	\end{equation*}
	can be extended to an open subset of $\R$ and is real-analytic there. We denote the vector space of such forms by $\mathcal Q_k(\Gamma)$.
\end{definition}

We next turn to the definition of higher-depth quantum modular forms.
\begin{definition}\label{defgen}
	A function $f:\mathcal Q\to\C$ ($\mathcal Q\subset \Q$) is called a {\it quantum modular form of depth $N\in\N$, weight $k \in \frac12\Z$, and quantum set $\mathcal Q$ for $\Gamma$} if for $\gamma=\left(\begin{smallmatrix}
	a & b\\ c & d
	\end{smallmatrix}\right)\in\Gamma$		
	\begin{equation*}
	f- f\big|_k \gamma \in  \bigoplus_j \mathcal Q_{\kappa_j}^{N_j}(\Gamma)\mathcal O(R),
	\end{equation*}
	where $j$ runs through a finite set, $\kappa_j\in \frac12\Z$, $N_j \in \N$ with $\max_j(N_j)=N-1$, $\mathcal Q_k^1(\Gamma) := \mathcal Q_k(\Gamma)$, $\mathcal Q_k^0(\Gamma):=1$, and $\mathcal Q_{k}^{N}(\Gamma)$ is the space of quantum modular forms of weight $k$ and depth $N$ for $\Gamma$.
\end{definition}
For $f_j\in S_{k_j}(\Gamma)$, the space of cusp forms of weight $k_j$ for $\Gamma$ with $k_j>\frac12$ define the {\it (non-holomorphic) Eichler integrals}
\begin{align*}
I_f(\tau)&:=\int_{-\overline{\tau}}^{i\infty}\frac{f(w)}{\left(-i(w+\tau)\right)^{2-k}}dw,\\
I_{f_1,f_2}(\tau) &:= \int_{-\overline{\tau}}^{i\infty} \int_{w_1}^{i\infty} \frac{f_1(w_1)f_2(w_2)}{(-i(w_1+\tau))^{2-k_1}(-i(w_2+\tau))^{2-k_2}} dw_2 dw_1,
\end{align*}
and the {\it errors of modularity}, for $\varrho \in \Q$
\begin{align*}
r_{f, \varrho}(\tau)&:=\int_{\varrho}^{i\infty}\frac{f(w)}{\left(-i(w+\tau)\right)^{2-k}}dw,\\
r_{f_1,f_2,\varrho}(\tau) &:= \int_{\varrho}^{i\infty} \int_{w_1}^{\varrho} \frac{f_1(w_1) f_2(w_2)}{(-i(w_1+\tau))^{2-k_1} (-i(w_2+\tau))^{2-k_2}} dw_2 dw_1.
\end{align*}
The next result is  \cite[Theorem 5.1]{BKM}.
\begin{theorem}\label{quantheorem}
	We have, for $\gamma=\left(\begin{smallmatrix}a&b\\c&d\end{smallmatrix}\right)\in\Gamma^\ast:=\left(\begin{smallmatrix}-1 & 0 \\ 0 & 1\end{smallmatrix}\right)\Gamma\left( \begin{smallmatrix}-1 & 0 \\ 0 & 1\end{smallmatrix}\right)$,
	\begin{equation*}
	 I_{f_1,f_2}(\tau)- I_{f_1,f_2}\mid_{k_1+k_2-4} \gamma (\tau) = r_{f_1,f_2,\tau,\frac{d}{c}}(\tau) + I_{f_1}(\tau)r_{f_2,\frac{d}{c}}(\tau).
	\end{equation*}
	Moreover  $r_{f_1,f_2, \frac{d}{c}}\in \mathcal O(\R)$.
\end{theorem}

\subsection{Definitions and notation}
\label{sec:Prelim:Def}

In this section we recall the construction of $Z(q)$ following \cite{BMM}, which is another invariant that is closely related to $\widehat{Z}_{\b{a}}(q)$ from \eqref{Gukov:ZqDef}.
Consider a tree $G$ with $N$ vertices labeled by integers $m_{jj}$, $ 1 \leq j \leq N$, which is called a {\it plumbing graph}. To this data we associate an $N \times N$ matrix $M=(m_{jk})_{1 \leq j,k \leq N}$, called its {\it linking} (or {\it plumbing}) {\it matrix}, such that $m_{jk}=-1$ if vertex $j$ is connected to vertex $k$ and zero otherwise.
We say that two plumbing matrices $M$ and $M'$ are {\it equivalent} if their underlying graphs are isomorphic, and there is a graph isomorphism that maps $M$ to $M'$. The first homology group of $M_3(G)$ (the plumbed $3$-manifold
constructed from $G$ and $M$) is
\begin{equation*}H_1(M_3(G),\mathbb{Z}) \cong {\rm coker}(M) =\mathbb{Z}^N/M\mathbb{Z}^N.\end{equation*}
If $M$ is invertible, then this group is finite and if $M \in \text{SL}_N(\mathbb{Z})$, then $H_1(M_3,\mathbb{Z})=0$; this is the case for the main results of this paper, as $M$ is positive definite and unimodular.

To each edge $j - k$  in $G$ we associate a rational function
\begin{equation}\label{f}
f(w_j,w_k):=\frac{1}{\big(w_j-w_j^{-1}\big)\big(w_k-w_k^{-1}\big)}
\end{equation}
and to each vertex $w_j$ a Laurent polynomial
\begin{equation}\label{g}
g(w_j):= \big(w_j - w_j^{-1}\big)^2.
\end{equation}

For a fixed tree $G$ and positive definite $M$, set
\begin{equation}
\label{E:ZqDef}
Z(q):= \frac{q^{\frac{-3N+\sum_{\nu=1}^Na_\nu}{2}}}{(2\pi i)^N} \text{PV} \int_{|w_j|=1} \prod_{j=1}^N g(w_j) \prod_{(k,\ell) \in E} f(w_k,w_\ell) \Theta_M(q;{{\b w}}) \frac{dw_j}{w_j},
\end{equation}
where we let $a_j:=m_{jj}$ for the vertex labels,  $w_j:=e^{2 \pi i z_j}$
\begin{equation*}
\Theta_M(q;\boldsymbol{w}):=\sum_{\boldsymbol{n}\in\mathbb Z^N}q^{\frac12 \b n^T M \b n}e^{2\pi i\b n^T M \b z}.
\end{equation*}
Note that we may write
\begin{equation}
\label{E:Theta2}	
\Theta_M(q;\b{w}) =\sum_{\b{m}\in M \mathbb Z^N}q^{\frac12 \b m^T M^{-1}\b m}e^{2\pi i\b m^T\b z}.
\end{equation}
The following result is given in Proposition 3.4 of \cite{BKM}.
\begin{proposition}
\label{P:uni}
If $M$ is unimodular, then  $Z(q)=\widehat{Z}_{\b \delta}(q^{2})$, where  $\widehat{Z}_{\b \delta}(q)$ is defined in (\ref{Gukov:ZqDef}).
\end{proposition}


\section{Some general construction}\label{sec:general}

In this section we construct an infinite family of quantum modular forms of depth two closely following the arguments in \cite{Br}.
Define
\begin{equation} \label{quantum-FS}
\mathcal Q_{\mathcal S,Q,\varepsilon}:=\left\{\tfrac{h}{k}\in\mathbb Q\,:\, \gcd(h,k)=1,\,k\in\N,\,  \sum_{\b{\alpha}\in\mathcal S} \varepsilon(\b \alpha) \sum_{\b{\ell}\pmod{k}} e^{2\pi i\frac hk KQ(\b \ell + \b \alpha)} = 0\right\}.
\end{equation}

\noindent We write $Q(\b n) =: \sigma_1n_1^2+2\sigma_2n_1n_2+\sigma_3n_2^2$,  and denote its discriminant by $D:=\sigma_1\sigma_3-\sigma_2^2.$
We also regularly use the relationship between the quadratic form and the associated bilinear form, namely
\begin{equation}
\label{E:QB}
Q(\b{x} + \b{y}) - Q(\b{x}) - Q(\b{y}) = B(\b{x}, \b{y}).
\end{equation}

\begin{theorem}\label{QMT}
	The functions $F_{\mathcal S,Q,\varepsilon}$ are quantum modular forms of depth two, weight one,  on some congruence subgroup containing $\Gamma(8 \cdot {\rm lcm}(\sigma_1,\sigma_3)KD)$, and quantum set $\mathcal Q_{\mathcal S,Q,\varepsilon}$.
\end{theorem}
Before proving Theorem \ref{QMT}, we require some auxiliary lemmas. Set
\begin{equation*}
\mathbb E_{\mathcal S,Q,\varepsilon}(\tau) :=\sum_{\b{\alpha}\in\mathcal S}\varepsilon(\b{\alpha})\mathbb F_{Q,\b{\alpha}}(\tau),
\end{equation*}
where
\begin{equation*}
\mathbb F_{Q,\b \alpha}(\tau) := \tfrac12\sum_{\b n\in\Z^2+\b{\alpha}} M_2\left(\kappa;(a_1n_1+a_2n_2,b_2n_2)\sqrt{Kv}\right)q^{-KQ(\b n)}
\end{equation*}
with
\begin{align*}
\kappa := \tfrac{\sigma_2}{\sqrt{D}}, \qquad a_1 := 2\sqrt{\sigma_1},\qquad a_2:=\tfrac{2\sigma_2}{\sqrt{\sigma_1}},\qquad b_2:=2\sqrt{\tfrac{D}{\sigma_1}}.
\end{align*}

We begin by determining the asymptotic expansions of these functions.

\begin{lemma} \label{Fthm}
If $\frac{h}{k}\in\mathcal Q_{\mathcal S,\varepsilon}$, then we have the asymptotic expansions (as $t\to 0^+$)
	\begin{align}\label{as}
	F_{\mathcal S,Q,\varepsilon}\left(\tfrac hk+\tfrac{it}{2\pi}\right) =: \sum_{m\geq 0} a_{h,k}(m) t^m,\qquad
	\mathbb E_{\mathcal S,Q,\varepsilon}\left(\tfrac hk+\tfrac{it}{2\pi}\right) = \sum_{m\geq 0} a_{-h,k}(m) (-t)^m.
	\end{align}
\end{lemma}
\begin{proof}
	For the proof we abbreviate
	\begin{equation*}
	F:=F_{\mathcal S, Q,\varepsilon},\qquad \mathbb E:=\mathbb E_{\mathcal S, Q,\varepsilon},\qquad \mathcal Q:=\mathcal Q_{\mathcal S,Q,\varepsilon}.
	\end{equation*}
		
	We first determine the asymptotic expansion of $F$ using the Euler-Maclaurin summation formula.
	We let $\b n\mapsto \b \ell + k\b n$ with $0\leq\b\ell\leq k-1$ (i.e., $0 \leq \ell_j \leq k-1$, $j
	\in \{1,2\}$), $\b n \in \IN_0^2$. The assumption that $K\mathcal S\subset \N^2$ implies that $\frac{h}{k}K Q(\b{\ell} + \b{\alpha} + k\b{n}) \equiv \frac{h}{k} K Q(\b{\ell} + \b{\alpha}) \pmod 1$, thus
	\begin{align*}
	F\left(\tfrac{h}{k}+\tfrac{it}{2\pi}\right)=\sum_{\b \alpha \in \mathcal S} \varepsilon(\b \alpha) \sum_{0\leq \b \ell \leq  k-1} e^{2\pi i\frac hk K Q(\b \ell+ \b \alpha)} \sum_{\b n\in\mathbb{N}_0^2+\frac{1}{k}(\b \ell+\b \alpha)} g\left(k\sqrt{t}\b n\right),
	\end{align*}
	where $g(\b x) := e^{-KQ(\b x)}$. The main term in Lemma \ref{EulerMcLaurin} is
	\begin{equation*}
	\frac{\mathcal I_g}{k^2t}  \sum_{\b{\alpha}\in\mathcal S} \varepsilon(\b \alpha) \sum_{0\leq \b \ell\leq  k-1} e^{2\pi i\frac{h}{k} K Q(\b \ell+\b{\alpha})}.
	\end{equation*}
	Using that $K\mathcal S\subset \mathbb N^2$  we may let $\b\ell$ run$\pmod{k}$. Since $\frac{h}{k}\in\mathcal Q$ the sum vanishes.
	
	The second term in Lemma \ref{EulerMcLaurin} yields
	\begin{equation}\label{Euler}
	-  \sum_{\b{\alpha}\in\mathcal S} \varepsilon(\b \alpha)\sum_{0\leq \b \ell\leq k-1} e^{2\pi i \frac hk K Q(\b \ell+\b{\alpha})} \sum_{n_2\geq 0} \frac{B_{n_2+1}\left(\frac{1}{k}(\ell_2+\alpha_2)\right)}{(n_2+1)!}
	\int_0^\infty g^{(0,n_2)} (x_1,0)dx_1 \left(k\sqrt{t}\right)^{n_2-1}.\hspace{-0.2cm}
	\end{equation}
	Making the change of variables $\b{\ell}\mapsto (k-1)(1,1)-\b{\ell}$ and using that $(1,1)-\b{\alpha}\in\mathcal S$ if $\b{\alpha}\in\mathcal S$, \eqref{Ber} yields that only the odd values of $n_2$ survive, and \eqref{Euler} becomes
	\begin{multline*}
	- \sum_{\b \alpha\in\mathcal S} \varepsilon(\b \alpha)\sum_{0\leq \b \ell\leq k-1} e^{2\pi i\frac hk KQ(\b \ell+\b \alpha)}\sum_{n_2\geq 0} \frac{B_{2n_2+2}\left(\frac{1}{k}(\ell_2+\alpha_2)\right)}{(2n_2+2)!}
	 \int_0^\infty g^{(0,2n_2+1)}(x_1,0)dx_1k^{2n_2} t^{n_2}.
	\end{multline*}
	
	In exactly the same way we obtain that the third term in  Lemma \ref{EulerMcLaurin} equals
	\begin{align*}
	- \sum_{\b \alpha\in\mathcal S} \varepsilon(\b \alpha)\sum_{0\leq \b \ell\leq k-1} e^{2\pi i\frac hk K Q(\b \ell+\b \alpha)}\sum_{n_1\geq 0} \frac{B_{2n_1+2}\left(\frac{1}{k}(\ell_1+\alpha_1)\right)}{(2n_1+2)!}
	 \int_0^\infty g^{(2n_1+1,0)}(0,x_2)dx_2k^{2n_1} t^{n_1}.
	\end{align*}

	For the final term in Lemma \ref{EulerMcLaurin} we obtain, pairing in exactly the same way
	\begin{multline*}
	 \sum_{\b \alpha\in\mathcal S}\varepsilon(\b \alpha) \sum_{0\leq \b \ell\leq  k-1} e^{2\pi i\frac hkKQ(\b \ell+\b \alpha)}
	 \sum_{\substack{n_1,n_2\geq 0 \\ n_1\equiv n_2\pmod{2}}}\hspace{-0.35cm} \frac{B_{n_1+1}\left(\frac1k(\ell_1+\alpha_1)\right)}{(n_1+1)!}\frac{B_{n_2+1}\left(\frac1k(\ell_2+\alpha_2)\right)}{(n_2+1)!}\\
	 \times g^{(n_1,n_2)}(0,0) \left(k\sqrt{t}\right)^{n_1+n_2}.
	\end{multline*}
	In particular we obtain that the asymptotic expansion of $F$ has the shape as claimed in \eqref{as}.
	
	We now turn to the asymptotic behavior of $\mathbb E$. We use \eqref{splitM} and let $M_2^*$ denote the function such that the $\sgn$ in \eqref{splitM} is replaced by $\sgn^*$, where $\sgn^*(x):= \sgn(x)$ if $x \in \R \setminus \{0\}$ and $\sgn^*(0):=1$. We obtain
	\begin{equation}\label{splitMstar}
	\begin{split}
		M_2^*(\kappa;a_1n_1+a_2n_2,b_2n_2)
		= E_2(\kappa;a_1n_1+a_2n_2,b_2n_2) + \sgn^*(n_1)\sgn^*(n_2)  \\&\hspace{-10.5cm}-\sgn^*(n_2)E\left(\tfrac{2}{\sqrt{\sigma_1}}(\sigma_1n_1+\sigma_2n_2)\right)
		- \sgn^*(n_1) E\left(\tfrac{2}{\sqrt{\sigma_3}} (\sigma_2 n_1+\sigma_3n_2)\right).
	\end{split}
	\end{equation}
Proceeding as above
	\begin{multline*}
	\mathbb E\left(\tfrac hk+\tfrac{it}{2\pi}\right)= \sum_{\b \alpha \in\mathcal S} \varepsilon(\b \alpha) \left(\sum_{0\leq \b \ell\leq k-1} e^{-2\pi i\frac hk KQ(\b{\ell}+\b \alpha)} \sum_{\b n \in \mathbb N_0^2+\b \alpha } G\left(k\sqrt{t}\b n\right) \right. \\
	\left. + \sum_{0\leq \b \ell\leq k-1} e^{-2\pi i\frac{h}{k} K\widetilde{Q}(\b{\ell}+\b{\alpha})}\sum_{\b n\in\mathbb N_0^2+\b{\alpha}}\widetilde{G}\left(k\sqrt{t}\b n\right)\right),
	\end{multline*}
where
	\begin{align*}
	\widetilde{Q}(x_1,x_2) := Q(-x_1,x_2)
	\end{align*}
	\begin{align*}	
	G(\b x):= \tfrac 12 M_2^* \left(\kappa; \sqrt{\tfrac{K}{2\pi}}(a_1x_1+a_2x_2, b_2x_2)\right) e^{KQ(\b x)},\qquad \widetilde{G}(x_1,x_2):=G(-x_1,x_2).
	\end{align*}
	
	We again use the Euler-Maclaurin summation formula. The main term in Lemma \ref{EulerMcLaurin} is
	\begin{equation*}
	\tfrac{4\mathcal I_G}{k^2t} \sum_{\b \alpha \in\mathcal S}  \varepsilon(\b \alpha) \sum_{0\leq \b \ell\leq k-1} e^{-2\pi i\frac hkKQ(\b \ell+\b \alpha)}
	+ \tfrac{4\mathcal I_{\widetilde G}}{k^2t}  \sum_{\b{\alpha}\in\mathcal S} \varepsilon(\b{\alpha})\sum_{0\leq \b \ell \leq k-1} e^{-2\pi i\frac{h}{k}K\widetilde Q(\b{\ell}+\b{\alpha})}=0
	\end{equation*}
	by conjugating the condition in $\mathcal Q$.
	
	The second term in Lemma \ref{EulerMcLaurin} is, pairing terms as before,
	\begin{multline*}
	-\sum_{\b \alpha\in\mathcal S}  \varepsilon(\b \alpha) \sum_{0\leq \b\ell\leq  k-1} e^{-2\pi i\frac{h}{k} K Q(\b \ell+\b\alpha)} \sum_{n_2\geq 0} \frac{B_{2n_2+2}\left(\frac1k(\ell_2+\alpha_2)\right)}{(2n_2+2)!} \\
	\times \int_0^\infty \left(G^{(0,2n_2+1)}(x_1,0) + \widetilde{G}^{(0,2n_2+1)}(x_1,0)\right)dx_1 \left(k^2 t\right)^{n_2}.
	\end{multline*}
	It is now straightforward to verify, as in \cite{BKM}, that
	\begin{align*}
	\int_0^\infty \left(G^{(0,2n_2+1)}(x_1,0) + \widetilde G^{(0,2n_2+1)}(x_1,0)\right) dx_1 &= (-1)^{n_2} \int_0^\infty g^{(0,2n_2+1)}(x_1,0) dx_1.
	\end{align*}	
	Via symmetry the third term in Lemma \ref{EulerMcLaurin} is treated in exactly the same way.
	
	The fourth term in Lemma \ref{EulerMcLaurin} is, pairing as before,
	\begin{multline*}
	 \sum_{\b \alpha\in \mathcal S}\! \varepsilon(\b \alpha) \!\sum_{0\leq \b \ell\leq k-1} \!e^{-2\pi i\frac hk KQ(\b \ell + \b \alpha)} \!\!\!\!\! \sum_{\substack{ n_1,n_2\geq 0 \\ n_1\equiv n_2\pmod{2}}}\!\!\!\!\! \frac{B_{n_1+1}\left(\frac1k(\ell_1+\alpha_1)\right)}{(n_1+1)!} \frac{B_{n_2+1}\left(\frac1k(\ell_2+\alpha_2)\right)}{(n_2+1)!} \\
	\times \ \left(G^{(n_1,n_2)}(0,0) + (-1)^{n_1+1} \widetilde G^{(n_1,n_2)}(0,0) \right) \left(k\sqrt{t}\right)^{n_1+n_2}.
	\end{multline*}
	It can now be shown that
	\begin{equation*}
	G^{(n_1,n_2)}(0,0) + (-1)^{n_1+1} \widetilde G^{(n_1,n_2)}(0,0) = i^{n_1+n_2} g^{(n_1,n_2)}(0,0).
	\end{equation*}
	Comparing terms gives the claim.
\end{proof}

Write $\mathcal{A} := K \mathcal{S}$,  and define
\begin{align*}
\mathcal{B} &:= \left\{ 0 \leq \b B < \sigma_1 K :  B_1 = \sigma_1 A_1 + \sigma_2 A_2 + \varrho \sigma_2 K, \, B_2 = A_2 + \varrho K, \text{ for some } \b{A}\! \in\! \mathcal{A}, \varrho\!\! \pmod{\sigma_1}\right\}\!,\\
\mathcal{C} &:= \left\{ 0 \leq \b C < \sigma_3 K :  C_1 = \sigma_2 A_1 + \sigma_3 A_2 + \varrho \sigma_2 K, \, C_2 = A_2 + \varrho K, \text{ for some } \b{A}\! \in \!\mathcal{A}, \varrho \!\! \pmod{\sigma_3}\right\}\!.
\end{align*}
The following lemma rewrites $\mathbb E$ as a two-dimensional theta integral, which is essential in order to calculate modular transformations.
\begin{lemma}\label{thetalemma}
	We have
	\begin{equation*}
	\mathbb E_{\mathcal S,Q,\varepsilon}(\tau) = -\frac{\sqrt{D}}{2\sigma_1K} \sum_{\b B\in\mathcal B} \varepsilon\left(\tfrac{B_1-\sigma_2B_2}{\sigma_1K}, \tfrac{B_2}{K}\right) I_{T_1, T_2}(\tau)
	-\frac{\sqrt{D}}{2\sigma_3K} \sum_{\b C\in\mathcal C} \varepsilon\left(\tfrac{C_2}{K}, \tfrac{C_2-\sigma_2 C_1}{\sigma_3K}\right) I_{U_1, U_2}(\tau),
	\end{equation*}
where
\begin{align*}
T_1(w) &:= \vartheta_1(\sigma_1K,B_1,\sigma_1K;2w),\qquad T_2(w) := \vartheta_1(\sigma_1K,B_2,\sigma_1K;2Dw),\\
U_1(w) &:= \vartheta_1(\sigma_3K,C_1,\sigma_3K;2w),
\qquad
U_2(w) := \vartheta_1(\sigma_3K,C_2,\sigma_3K;2Dw).
\end{align*}
\end{lemma}

\begin{proof}
Using \eqref{Mrep} we obtain
\begin{align*}
&M_2\left(\kappa;(a_1n_1+a_2n_2,bn_2)\sqrt{Kv}\right) q^{-KQ(\b n)} \\
&\hspace{2.5cm}=-\frac{2\sqrt{D}}{\sigma_1} (\sigma_1n_1+\sigma_2n_2)n_2 \int_{-K\overline{\tau}}^{i\infty} \int_{w_1}^{i\infty} \frac{e^{\frac{2\pi i}{\sigma_1}(\sigma_1n_1+\sigma_2n_2)^2 w_1 + \frac{2\pi iD n_2^2}{\sigma_1}w_2}}{\sqrt{-i(w_1+K\tau)}\sqrt{-i(w_2+K\tau)}}dw_2dw_1 \\
&\hspace{2.9cm} -\frac{2\sqrt{D}}{\sigma_3} (\sigma_2n_1+\sigma_3n_2)n_1 \int_{-K\overline{\tau}}^{i\infty} \int_{w_1}^{i\infty} \frac{e^{\frac{2\pi i}{\sigma_3}(\sigma_2n_1+\sigma_3n_2)^2 w_1 + \frac{2\pi iD n_1^2}{\sigma_3}w_2}}{\sqrt{-i(w_1+K\tau)}\sqrt{-i(w_2+K\tau)}}dw_2dw_1.
\end{align*}
This yields
\begin{multline*}
\mathbb E_{\mathcal S,Q,\varepsilon}(\tau) = -\frac{K\sqrt{D}}{\sigma_1} \int_{-\overline{\tau}}^{i\infty} \int_{w_1}^{i\infty} \frac{\theta_1(\b w)}{\sqrt{-i(w_1+\tau)}\sqrt{-i(w_2+\tau)}} dw_2dw_1 \\
- \frac{K \sqrt{D}}{\sigma_3} \int_{-\overline{\tau}}^{i\infty} \int_{w_1}^{i\infty} \frac{\theta_2(\b w)}{\sqrt{-i(w_1+\tau)}\sqrt{-i(w_2+\tau)}} dw_2dw_1,
\end{multline*}
where
\begin{align*}
\theta_1(\b w) &:= \sum_{\b{\alpha}\in\mathcal S} \varepsilon(\b \alpha) \sum_{\b n\in \Z^2 + \b \alpha} (\sigma_1n_1+\sigma_2n_2)n_2 e^{\frac{2\pi iK}{\sigma_1}(\sigma_1n_1+\sigma_2n_2)^2 w_1 + \frac{2\pi iDK}{\sigma_1}n_2^2w_2},\\
\theta_2(\b w) &:=  \sum_{\b \alpha\in \mathcal S} \varepsilon(\b \alpha) \sum_{\b n\in\Z^2+\b \alpha} (\sigma_2n_1+\sigma_3n_2)n_1 e^{\frac{2\pi i K}{\sigma_3}(\sigma_2n_1+\sigma_3n_2)^2w_1+ \frac{2 \pi i DK}{\sigma_3}n_1^2w_2}.
\end{align*}

We now rewrite the $\theta_j$ in terms of the Shimura theta functions.
Letting $\b n\mapsto \frac{\b n}{K}$, we obtain
\begin{align*}
\theta_1(\b w) = \tfrac{1}{K^2} \sum_{\b A\in\mathcal A} \varepsilon\left(\tfrac{\b A}{K}\right) \sum_{\b n\equiv \b A\pmod{K}} (\sigma_1n_1+\sigma_2n_2)n_2 e^{\frac{2\pi i}{\sigma_1 K} (\sigma_1n_1+\sigma_2n_2)^2w_1+\frac{2\pi iD}{\sigma_1K} n_2^2w_2}.
\end{align*}
Set $\nu_1:=\sigma_1n_1+\sigma_2n_2$ and $\nu_2:=n_2$,
so that $n_1 = \frac{\nu_1 - \sigma_2 \nu_2}{\sigma_1}$. Plugging in the restrictions on $\b{n}$ yields
\begin{align*}
&\nu_2 \equiv A_2 + \varrho K \pmod{\sigma_1 K}\quad \text{ for }  0 \leq \varrho \leq \sigma_1 - 1, \\
&\nu_1 = \sigma_1 n_1 + \sigma_2 n_2 \equiv \sigma_1 A_1 + \sigma_2 A_2 + \varrho \sigma_2 K \pmod{\sigma_1 K}.
\end{align*}
This shows that  $\b{\nu} \in \mathcal{B}.$  Furthermore, if $\b \alpha \in \mathcal{A}$, there exists a corresponding $\b{B} \in \mathcal{B}$ such that
\begin{equation*}
\b \alpha = \left(\tfrac{A_1}{K}, \tfrac{A_2}{K}\right)
\equiv \left(\tfrac{B_1 - \sigma_2 B_2}{\sigma_1 K}, \tfrac{B_2}{K}\right) \pmod{1}.
\end{equation*}

Overall, we therefore have
\begin{align*}
\theta_1(\b w)&= \frac{1}{K^2}  \sum_{\b B \in \mathcal{B}}\! \varepsilon \! \left(\tfrac{B_1-\sigma_2 B_2 }{\sigma_1K}, \tfrac{B_2}{K} \right) \!\sum_{\nu_1 \equiv B_1 \pmod{\sigma_1 K}}\!\! \nu_1 e^{\frac{2 \pi i  \nu_1^2 w_1}{\sigma_1 K}}\!\! \sum_{\nu_2 \equiv B_2 \pmod{\sigma_1 K}}\!\! \nu_2 e^{\frac{2 \pi i D \nu_2^2 w_2}{\sigma_1K}}\\
&=\frac{1}{K^2} \sum_{\b B\in \mathcal{B}} \varepsilon \left( \tfrac{B_1-\sigma_2 B_2}{\sigma_1 K}, \tfrac{B_2}{K}\right) \vartheta_1 \left(\sigma_1K, B_1, \sigma_1 K; 2w_1 \right) \vartheta_1 \left(\sigma_1 K, B_2, \sigma_1 K; 2Dw_2\right).
\end{align*}

In the same way, by setting $\nu_1 := \sigma_2 n_1 + \sigma_3 n_2$ and $\nu_2 := n_1$, we can show that
\begin{align*}
\theta_2(\b w) &= \frac{1}{K^2}  \sum_{\b C\in\mathcal{C}} \varepsilon\left(\tfrac{C_2}{K}, \tfrac{C_2-\sigma_2C_1}{\sigma_3 K}\right) \vartheta_1(\sigma_3K,C_1,\sigma_3K;2w_1)\vartheta_1(\sigma_3 K,C_2,\sigma_3K;2Dw_2).
\end{align*}
\vskip-2.75em
\end{proof}
\vskip1.25em

We are now ready to prove Theorem \ref{QMT}.
\begin{proof}[Proof of Theorem \ref{QMT}]

Suppose that $f$ is one of the theta functions from Lemma \ref{thetalemma} and $\gamma\in \Gamma(8 \cdot {\rm lcm}(\sigma_1, \sigma_3) KD)$. Then the transformation \eqref{Shimura1} implies (after a short calculation) that $f \!\! \mid_{\frac32} \! \!  \gamma  = \!  f$. The theorem statement now follows from Lemmas \ref{Fthm} and \ref{thetalemma}, and Theorem \ref{quantheorem}.
\end{proof}

\section{A family with quantum set $\Q$ and unimodular matrices} \label{sec:family}

In this section we construct a family of depth two quantum modular forms with quantum set $\mathbb{Q}$. Let $N_1,N_2 \in 2\N$ and write $L := \gcd(N_1, N_2)$, $N_1 := L R_1, N_2 := L R_2,$ so that $\gcd(R_1,R_2) = 1$. Set $Q(\b{n}) = \sigma_1 n_1^2 + 2 \sigma_2 n_1 n_2 + \sigma_3 n_2^2$. We assume the factorizations $\sigma_1 = R_1 \mu_1 ,$ with $\gcd(R_1, \mu_1) = 1$, and $\sigma_3 = R_2 \mu_3$, with $\gcd(\mu_3,R_2) = 1$. Moreover we assume that $2 \sigma_2 = L R_1 R_2= {\rm lcm}(N_1, N_2)$ and that $\gcd(\mu_1, \mu_3)$ consists of at most one odd prime factor, and always satisfies $\gcd(L, \gcd(\mu_1, \mu_3)) = 1$. If $4 \nmid L$, then we also require that exactly one of $R_1, R_2,\mu_3$ is even.
Set, with $r_1,r_2,s_1,s_2\in\mathbb N$ satisfying $\gcd(r_j, N_j) = \gcd(s_j, N_j) = 1$, $r_j^2\equiv s_j^2\pmod{2N_j}$,
\begin{align}
\label{E:S1S2def}
\mathcal S_1&:=\left\lbrace \left(\tfrac{r_1}{N_1},\tfrac{r_2}{N_2}\right),\left(1-\tfrac{r_1}{N_1},\tfrac{r_2}{N_2}\right),\left(\tfrac{r_1}{N_1},1-\tfrac{r_2}{N_2}\right),\left(1-\tfrac{r_1}{N_1},1-\tfrac{r_2}{N_2}\right)\right.,\\ \notag
&\left.\qquad\, \left(\tfrac{s_1}{N_1},\tfrac{s_2}{N_2}\right),\left(1-\tfrac{s_1}{N_1},\tfrac{s_2}{N_2}\right),\left(\tfrac{s_1}{N_1},1-\tfrac{s_2}{N_2}\right),\left(1-\tfrac{s_1}{N_1},1-\tfrac{s_2}{N_2}\right)\right\rbrace ,\\ \notag
\mathcal S_2&:=\left\lbrace \left(\tfrac{r_1}{N_1},\tfrac{s_2}{N_2}\right),\left(1-\tfrac{r_1}{N_1},\tfrac{s_2}{N_2}\right),\left(\tfrac{r_1}{N_1},1-\tfrac{s_2}{N_2}\right),\left(1-\tfrac{r_1}{N_1},1-\tfrac{s_2}{N_2}\right),\right.\\ \notag
&\left.\qquad\, \left(\tfrac{s_1}{N_1},\tfrac{r_2}{N_2}\right),\left(1-\tfrac{s_1}{N_1},\tfrac{r_2}{N_2}\right),\left(\tfrac{s_1}{N_1},1-\tfrac{r_2}{N_2}\right),\left(1-\tfrac{s_1}{N_1},1-\tfrac{r_2}{N_2}\right)\right\rbrace .
\end{align}

We define
	\begin{equation*}
	\mathcal{Z}_{Q, \b r, \b s}(q):= \sum_{j \in \{1,2\}} (-1)^{j+1} \sum_{\b \alpha \in \mathcal{S}_j}\sum_{\b n \in \N_0^2} q^{L Q(\b n +\b \alpha)}= F_{\mathcal S,Q, \varepsilon}\left( \tfrac{\tau}{R_1R_2} \right),
	\end{equation*}
where  $\mathcal S:= \mathcal S_1 \cup \mathcal S_2$ and $\varepsilon(\b \alpha) :=(-1)^{j+1}$ if $\b \alpha \in \mathcal S_j$.
We  see in the proof of Theorem \ref{main-thm} that the assumptions imply that the asymptotic expansion of $\mathcal{Z}_{Q, \b r, \b s}(q)$ consists of several leading terms with identical Gauss sums that always cancel, and thus the series converges for all $\Q$.
	
	\begin{theorem} \label{main-thm}
		Under the assumption above, the function $\mathcal{Z}_{Q, \b r, \b s}(q)$ is a quantum modular form of depth two, weight one, group $\Gamma \left( 8\cdot \text{\rm lcm}(\sigma_1, \sigma_2) L R_1R_2\right)$, and quantum set $\mathbb{Q}$.
	\end{theorem}	
\begin{proof}
	Note that the conditions of Theorem \ref{QMT} are satisfied. We are left to show that we have quantum set $\mathbb{Q}$, which follows if we show that
		\begin{equation}
		\label{E:ellsum}
		\sum_{j\in\{1,2\}}(-1)^{j+1}\sum_{\b{\alpha}\in\mathcal S_j} \sum_{\b{\ell} \pmod{k}} e^{2 \pi i \frac{h}{k} L Q\left(\b{\ell} + \b{\alpha}\right)}=0.
		\end{equation}
		Write $L = 2^\Lambda L_1, k = g k_1$, where $L_1, k_1$ are odd and where $g := \gcd(k,L)$. We claim that the sum on $\b{\ell}$ vanishes unless $\gcd(LR_1R_2,k_1)=1$ and $g\in\{1,2\}$. For this we first consider the (one-dimensional) Gauss sum in $\ell_1$, which is $(a_j:=N_j \alpha_j)$
		\begin{equation}\label{E:ell1Gauss}
			\sum_{\ell_1\pmod{k}} e^{2\pi i\frac{h}{k}\left(LR_1\mu_1\ell_1^2 + \left(2\mu_1a_1+L^2R_1R_2\ell_2+LR_1a_2\right)\ell_1\right)}.
		\end{equation}
	The linear term reduces to $2 \mu_1 a_1 \pmod{R_1}$, and $\mu_1 a_1$ is coprime to $R_1$ by assumption. Thus by Proposition \ref{P:G=0} the expression in \eqref{E:ell1Gauss} is zero if $\gcd(R_1, k_1) > 1$. Similarly, the linear term reduces to $2 \mu_1 a_1 \pmod{L}$. The Gauss sum \eqref{E:ell1Gauss} vanishes if $g > 1$ and $g \nmid 2 \mu_1$. Now write an alternative Gauss sum by grouping the $\ell_2$ terms in \eqref{E:ellsum}, obtaining an analogous version of \eqref{E:ell1Gauss}. As before, this immediately shows that \eqref{E:ellsum} is zero if $\gcd(R_2, k_1) > 1$, and also vanishes if $g > 1$ and $g \nmid 2 \mu_3$. If $g > 1$, then the only way the sum fails to vanish is if $g \mid \gcd (2 \mu_1, 2 \mu_3),$ which implies that $g = 2$ by assumption. This shows that \eqref{E:ellsum} vanishes if $4 \mid L$.

Next, assuming $g = 2$, $4 \nmid L$, and $4 \mid k$, we also show that \eqref{E:ellsum} vanishes in this case.  Recalling the corresponding assumptions on the $R_j$ and $\mu_j$, one possibility is that $2|R_1$ and $2\nmid R_2\mu_1\mu_2$ (or the analogous condition with $\ell_1$ and $\ell_2$ swapped if necessary). Then $4$ divides the factor in front of $\ell_1^2$ in \eqref{E:ell1Gauss}, and the linear term is congruent to $2$ modulo $4$ since $a_1$ is odd. The sum therefore vanishes by Proposition \ref{P:G=0}. Otherwise the condition on $R_j$ and $\mu_j$ is that $2\nmid R_1R_2\mu_1, 2 \mid \mu_3$, and we again consider the analog of \eqref{E:ell1Gauss} for the sum in $\ell_2$. Now $4$ divides the coefficient in front of $\ell_2^2$, and the linear term is congruent to $2\pmod{4}$ so Proposition \ref{P:G=0} again applies.

		We next assume that $\gcd(L R_1 R_2, k_1) = 1$, and $g \in\{1,2\}$ and prove that the sum on $\b{\ell}$ in \eqref{E:ellsum} is the same for all choices of $\b{\alpha}$. We note that the multiplicative inverses $\overline{N_j} \pmod{k_1}$ exist. Using \eqref{E:QB}, we write
		\begin{multline}
		\label{E:Qshift}
		\frac{h}{k} L \left(Q\left(\b{\ell} + \b \alpha \right) - Q\left(\b{\ell} + \left(\overline{N_1}a_1, \overline{N_2}a_2\right)\right)\right) \\
		= \frac{hL}{k} \left(Q(\b{\alpha}) - Q\left(\overline{N_1}a_1, \overline{N_2}a_2\right) + B\left(\b{\ell}, \b{\alpha}\right) - B\left(\b{\ell}, \left(\overline{N_1}a_1, \overline{N_2}a_2\right)\right)\right).
		\end{multline}
		Since $B(\b{\ell}, \b{\alpha}) - B(\b{\ell},(\overline{N_1}a_1, \overline{N_2}a_2)) \equiv 0 \pmod{k_1}$  by construction, \eqref{E:Qshift} implies that
		\begin{equation*}
		\frac{hL}{k_1} Q\left(\b{\ell} + \b{\alpha} \right)
		\equiv
		\frac{hL}{k_1} \bigg(Q\left(\b{\ell} + \left(\overline{N_1}a_1, \overline{N_2}a_2\right)\right)
		+ Q(\b{\alpha}) - Q\left(\overline{N_1}a_1, \overline{N_2}a_2\right) \bigg) \pmod{1}.
		\end{equation*}
		We now calculate
		\begin{equation*}
		\frac{hL}{k_1} \left(Q(\b{\alpha}) - Q\left(\overline{N_1}a_1, \overline{N_2}a_2\right) \right)
		= \frac{h}{k L R_1 R_2} X,
		\end{equation*}
		where $X:=R_2 \mu_1 a_1^2 + L R_1 R_2 a_1 a_2 + R_1 \mu_3 a_2^2 - N_1 N_2Q(\overline{N_1}a_1, \overline{N_2}a_2)$.

If $p$ is an odd prime such that $p^{\lambda}$ exactly divides $L R_1 R_2$, then the assumptions on the parameters easily imply that
\begin{equation*}
X \equiv R_2 \mu_1 a_1^2 + R_1 \mu_3 a_2^2 \pmod{p^{\lambda}}
\end{equation*}
is independent from $\b \alpha$.

Finally, suppose that $2^\lambda$ exactly divides $L R_1 R_2$. Then the final congruence is
		\begin{equation*}
		X \equiv R_2 \mu_1 a_1^2 + R_1 R_2 + R_1 \mu_3 a_2^2 \pmod{2^{\lambda} g},
		\end{equation*}
which is independent from $\b \alpha$ due to the assumption that $r_j^2 \equiv s_j^2 \pmod{2^{\lambda+1}}.$	
		
		Therefore the sum on $\b\ell$ in \eqref{E:ellsum} equals
		\begin{equation*}
		e^{2\pi i \frac{hX}{kL R_1 R_2}} \sum_{\b{\ell} \pmod{k}} e^{2 \pi i \frac{h}{k} L Q\left(\b{\ell} + \left(\overline{N_1}a_1,\, \overline{N_2}a_2\right)\right)}
		= e^{2\pi i \frac{hX}{L R_1 R_2}} \sum_{\b{\ell} \pmod{k}} e^{2 \pi i \frac{h}{k} L Q\left(\b{\ell}\right)}
		\end{equation*}
		by shifting $\b{\ell}$; this overall expression is now clearly independent from choice of $\b{\alpha}$.\qedhere
\end{proof}

\section{Classification of positive unimodular ${\tt H}$-matrices and the proofs of Theorem \ref{39} and Theorem \ref{T:mainquantum}}\label{sec:class}

\subsection{Proof of Theorem \ref{39}}
Let
\begin{equation}
\label{E:Mb}
M=M(b_1,b_2,b_3,b_4,b_5,b_6):=\left(
\begin{smallmatrix}
b_1 & 0 & -1 & 0 & 0 & 0 \\
0 & b_2 & -1 & 0 & 0 & 0 \\
-1 & -1 & b_3 & -1 & 0 & 0 \\
0 & 0 & -1 & b_4 & -1 & -1 \\
0 & 0 & 0 & -1 & b_5 & 0 \\
0 & 0 & 0 & -1 & 0 & b_6 \\
\end{smallmatrix}
\right).
\end{equation}
In this section, we classify all positive, unimodular (PU) matrices $M$ with the additional property that $b_j\geq 2$ ($j\in\{1,2,5,6\}$).
The determinant of $M$ can be written as follows:
\begin{align*}
D&=D(b_1,b_2,b_3,b_4,b_5,b_6) :=\det(M)\\
&= b_1b_2b_3b_4b_5b_6 - b_1b_2b_3b_5 - b_1b_2b_3b_6 - b_1b_2b_5b_6 -b_1b_4b_5b_6 - b_2b_4b_5b_6 + (b_1+b_2)(b_5+b_6) \\
& = b_1b_2b_5b_6\left(\left(b_3 - \tfrac{1}{b_1} - \tfrac{1}{b_2}\right)\left(b_4 - \tfrac{1}{b_5} - \tfrac{1}{b_6}\right) - 1\right).
\end{align*}
The goal of this section is to show the following.

\begin{proposition}
\label{P:Mfin}
If $M(b_1,b_2,b_3,b_4,b_5,b_6)$ is a {\rm PU} matrix with $b_j\geq 2$ ($j\in\{1,2,5,6\}$), then (up to equivalence)
\begin{equation*}
b_1 \leq 23, \; b_2 \leq 133, \; 2 \leq b_3 \leq 7, \; b_4 = 1, \; b_5 \leq 13, \; b_6 \leq 97.
\end{equation*}
In particular, there are finitely many {\rm PU} matrices $M$.
\end{proposition}
This then enables us to prove Theorem \ref{39}.
\begin{proof}[Proof of Theorem \ref{39}]
Proposition \ref{P:Mfin} together with a computer search quickly shows there are $312$ PU matrices. Since the group of automorphisms of an ${\tt H}$-graph is $\mathbb{Z}_2 \times \mathbb{Z}_2 \times \mathbb{Z}_2$, we have $39$ equivalence classes of such matrices; these are listed in the appendix. This gives the claim.	
\end{proof}

We now prove the main statement of this section, namely Proposition  \ref{P:Mfin}.
\begin{proof} [Proof of Proposition \ref{P:Mfin}]
It is clear that $\gcd(b_1, b_2) \mid D$, thus $b_1$ and $b_2$  must be coprime and without loss of generality we may assume $b_1 < b_2$ and $b_5 < b_6$. This further implies that $b_1 b_2, b_5 b_6 \geq 6$, and $\frac{1}{b_1} + \frac{1}{b_2} \leq \frac56$. Since $b_3, b_4 \geq 2$ we therefore have
\begin{equation*}
D \geq b_1 b_2 b_5 b_6 \left(\tfrac76 \cdot \tfrac76 - 1\right)
= \tfrac{13}{36} b_1 b_2 b_5 b_6 > 1
\end{equation*}
thus $M$ is not unimodular. Furthermore, the fact that $1 - \frac{1}{b_1} - \frac{1}{b_2} < 1$ immediately shows that if $b_3 = b_4 = 1$, then $\det (M) < 0$. Thus without loss of generality we assume that $b_4 = 1$ and $b_3\neq 1$. If $b_3 \geq 8$, then
\begin{equation*}
D > b_1 b_2 b_5 b_6 \left((b_3 - 1)  \tfrac16 - 1\right)
\geq b_1 b_2 b_5 b_6  \frac16 > 1.
\end{equation*}
Thus we must have $b_3 \leq 7$.

Now suppose that $b_5 \geq 14$. Then, since $b_6>b_5$,
\begin{align*}
D\geq
b_1 b_2 b_5 b_6 \left(\tfrac76  \left(1 - \tfrac{1}{14} - \tfrac{1}{15}\right) - 1\right) = 2 \cdot 3 \cdot 14 \cdot 15 \tfrac{1}{180} > 1.
\end{align*}
Thus we must have $b_5 \leq 13$.

The remaining bounds require a case by case analysis based on the values of $b_5$. If $b_5 = 2$, then for $b_2=2$, $D \leq 0$, thus we must have $b_3 \geq 3$. If $b_6 \geq 28$, then
\begin{equation*}
D\geq 6 \cdot 2 \cdot 27 \left( \left(3-\tfrac12-\tfrac13\right) \left(1 - \tfrac12 - \tfrac{1}{28}\right) - 1\right)
\geq  2.
\end{equation*}
We therefore conclude that $b_6 \leq 27$. However, in order to have $D$ positive we also need
$$
3 \left(\tfrac12 - \tfrac{1}{b_6}\right) > 1,
$$
which implies that $b_6 \geq 7$.

We next determine the possible values of $b_1$. In order to have $D = 1$, it must be true that $D > 0$, thus
\begin{align}
\label{E:b5=2}
\quad b_3 - \tfrac{1}{b_1} - \tfrac{1}{b_2}  > \left(\tfrac12 - \tfrac{1}{b_6}\right)^{-1}.
\end{align}
Now suppose that $3 \leq b_3 \leq 7$ and $14 \leq b_6 \leq 27$ are fixed.  Now suppose that $b_1 \geq 11$. Then
$$
D \geq 11 \cdot 12 \cdot 2 \cdot 7 \left(\left(3 - \tfrac{1}{11} - \tfrac{1}{12}\right)\left(\tfrac12 - \tfrac17\right) - 1\right) = 17,
$$
so we must have $b_1 \leq 10.$

In this case a Maple calculation shows that the right-side is at most $5$ (which occurs for $b_3 = 3$ and $b_6 = 7$), and thus all $b_1 > 10$ are not possible; in other words, we must have $b_1 \leq 10$. To complete this case, we now consider fixed $2 \leq b_1 \leq 10, 3 \leq b_3 \leq 7,$ and $3 \leq b_6 \leq 27$. If there is a solution, then following \eqref{E:b5=2}, it must be for the minimal value of $b_2$ such that
\begin{equation}
\label{E:b2range}
b_2 > - \left(\left(\tfrac12 - \tfrac{1}{b_6}\right)^{-1} - b_3 + \tfrac{1}{b_1}\right)^{-1}.
\end{equation}
A Maple search shows that the maximum value of the right-side is 30 (which occurs with $b_1 = 6, b_3 = 3$ and $b_6 = 7$), $b_2 \leq 31$.

Next, let $b_5 = 3$. If $b_3 \geq 4$, then
\begin{equation*}
D \geq 6 \cdot 3 \cdot 4 \left(\left(4-\tfrac12-\tfrac13\right) \left(1 - \tfrac13 - \tfrac14\right) - 1\right)
= 23.
\end{equation*}
Thus $b_3 \leq 3$, and we begin with $b_3 = 3$. Very similar calculations show, in turn, that $b_6 \leq 5$, and $b_1 \leq 3$. As in \eqref{E:b2range}, checking
\begin{equation*}
b_2 > - \left(\left(\tfrac23 - \tfrac{1}{b_6}\right)^{-1} - b_3 + \tfrac{1}{b_1}\right)^{-1}.
\end{equation*}
in these ranges now gives a maximum right-side value of $10$ (with $b_1 = 2, b_3 = 3$, and $b_6 = 4$), then $b_2 \leq 11$.

For the case $b_5 =3$ and $b_3 = 2$, if $b_6 \leq 6$, then
\begin{equation*}
D = b_1 b_2 b_5 b_6 \left(\left(2 - \tfrac{1}{b_1} - \tfrac{1}{b_2}\right)\left(1-\tfrac13 - \tfrac{1}{b_6}\right) - 1\right)
< b_1 b_2 b_5 b_6 \left(2 \cdot \tfrac12 - 1 \right) = 0,
\end{equation*}
and thus we must have $b_6 \geq 7$. However, in order for $D > 0$, it also must be true that
\begin{align}
\label{E:b5=3b6bound}
2 - \tfrac{1}{b_1} - \tfrac{1}{b_2} > \left(\tfrac23 - \tfrac{1}{b_6}\right)^{-1} > \tfrac32.
\end{align}
The largest values of $b_6$ occurs when the left side is as close to $\frac32$ as possible (while being larger, so $b_1 \geq 3$), which occurs for $b_1 = 3$ and $b_2 = 7$ (and then $2 - \frac13 - \frac17 = \frac{32}{21}$).  Plugging in to \eqref{E:b5=3b6bound}, this implies that the first inequality holds for $b_6 > 96$, and again by monotonicity, this gives the bound $b_6 \leq 97$.

Furthermore, if $b_1 \geq 24$, then
\begin{equation*}
D \geq 24 \cdot 25 \cdot 3 \cdot 7 \left(\left(2 - \tfrac{1}{24} - \tfrac{1}{25}\right)\left(1-\tfrac13 - \tfrac17\right) - 1\right)
= 61,
\end{equation*}
thus we must have $b_1 \leq 23$. Finally, checking
\begin{equation*}
b_2 > - \left(\left(\tfrac23 - \tfrac{1}{b_6}\right)^{-1} - 2 + \tfrac{1}{b_1}\right)^{-1}
\end{equation*}
over the ranges $3 \leq b_1 \leq 23$, and $7 \leq b_6 \leq 97$ shows that the right-side is at most 132 (which occurs at $b_1 = 12$ and $b_6 = 7$), so $b_2 \leq 133$.

For the remaining values $4 \leq b_5 \leq 13$, we proceed similarly. First, if $b_3 \geq 3$, then
\begin{equation*}
D \geq 2 \cdot 3 \cdot 4 \cdot 5 \left(\tfrac{13}{6} \left(1 - \tfrac14 - \tfrac15\right) - 1\right) =23,
\end{equation*}
thus we must have $b_3 = 2$. Furthermore, if $b_1 \geq 11$, then
\begin{align*}
D \geq 11 \cdot 12 \cdot 4 \cdot 5 \left(\left(2 - \tfrac{1}{11} - \tfrac{1}{12}\right) \cdot \tfrac{11}{20} - 1\right) = 11,
\end{align*}
thus $b_1 \leq 10$.

Now we bound $b_6$ as in the previous case. For example, if $b_5 = 4$, then $D > 0$ requires that
\begin{equation*}
2 - \tfrac{1}{b_1} - \tfrac{1}{b_2} > \left(\tfrac34 - \tfrac{1}{b_6}\right)^{-1} > \tfrac43.
\end{equation*}
This is only possible if $\frac{1}{b_1} + \frac{1}{b_2} < \frac23$, and the largest value of $b_6$ occurs when the sum is as close as possible to $\frac23$. This occurs with $b_1 = 2, b_2 = 7$, which implies that $b_6 \leq 77$. Repeating the argument for $b_5 \geq 5$ never gives a larger range for $b_6$ (and $b_5 \geq 8$ can be treated as a single case, since then the maximal case is always $\frac12 + \frac13 < \frac{b_5 - 2}{b_5 - 1}$). Finally, plugging in $b_1 \leq 10, 4 \leq b_5 \leq 13$, and $b_6 \leq 77$ to
\begin{equation*}
b_2 > - \left(\left(1 - \tfrac{1}{b_5} - \tfrac{1}{b_6}\right)^{-1} - 2 + \tfrac{1}{b_1}\right)^{-1}
\end{equation*}
gives the bound $b_2 \leq 71$.
\end{proof}

\subsection{Calculation of $Z(q)$ and the proof of Theorem \ref{T:mainquantum}.}

Let $M$ be as in \eqref{E:Mb}, with inverse matrix $M^{-1} = (\ell_{jk})_{1 \leq j,k \leq 6}$. We need the central $2 \times 2$ sub-matrix of $M^{-1}$, which we write as
$$
A :=  \begin{pmatrix} \ell_{33} & \ell_{34} \\ \ell_{43} & \ell_{44} \end{pmatrix}=\left(
\begin{array}{cc}
 {b_1} {b_2} ({b_4} {b_5} {b_6}-b_5-{b_6}) &  {b_1}
   {b_2} {b_5} {b_6} \\
  {b_1} {b_2} {b_5} {b_6} & \frac{{b_5} {b_6} ({b_1}
   {b_2} {b_5} {b_6}+1)}{{b_4} {b_5} {b_6}-b_5-{b_6}} \\
\end{array}
\right).
$$
In order to write $Z(q)$ as a double series of the type found in Section \ref{sec:family}, we use a linear algebra identity, which can be verified by a Maple computation.
\begin{lemma}
\label{L:M=A+c}
If $\b{r} = (\varepsilon_1, \varepsilon_2, 2n_1 +1, 2n_2 + 1, \varepsilon_5, \varepsilon_6)^T$ with $n_1,n_2 \in \Z$ and $\varepsilon_j \in \{\pm1\}$, then
$$
\frac12 \b{r}^T M^{-1} \b{r} = \frac12 \begin{pmatrix} 2n_1 + 2 \alpha_1,\, 2n_2 + 2 \alpha_2 \end{pmatrix} A \begin{pmatrix} 2n_1 + 2 \alpha_1 \\ 2n_2 + 2 \alpha_2 \end{pmatrix} + c ,
$$
where
$$
\b{\alpha} = \b{\alpha}(\varepsilon) = \begin{pmatrix} \alpha_1 \\ \alpha_2 \end{pmatrix}
:= \frac12 \begin{pmatrix} 1 + \frac{\varepsilon_1}{b_1} + \frac{\varepsilon_2}{b_2} \\ 1 + \frac{\varepsilon_5}{b_5} + \frac{\varepsilon_6}{b_6}\end{pmatrix},\qquad
c := \tfrac12\left(\tfrac{1}{b_1} + \tfrac{1}{b_2} + \tfrac{1}{b_5} + \tfrac{1}{b_6}\right). $$
\end{lemma}
\begin{remark}
Importantly, note that $c$ is independent of the $\varepsilon_j$'s.
\end{remark}

We can now evaluate $Z(q)$ for any positive unimodular $M$.
\begin{proposition}\label{prop:Zq}
With $\mathcal{S} := \left\{ \b{\alpha}(\varepsilon) \right\}$, we have
\begin{align} \label{qqq}
Z(q)= \frac{q^{-9+\frac{{\rm tr}(M)}{2} + c}}{4} \sum_{\b \alpha \in \mathcal S } (-1)^{j+1} \sum_{\b n\in \mathbb{Z}^2} {\rm sgn}^*(n_1) {\rm sgn}^*(n_2) q^{Q_1(\b n +\b \alpha)},
\end{align}
where $Q_1(\b n):=\frac{1}{2} {\b m}^T M^{-1}{\b m}$, with ${\b m}:=(0,0,2n_1, 2n_2,0,0)^T$.
\end{proposition}

\begin{proof}
	 An application of formula (\ref{E:ZqDef}) for the ${\tt H}$-graph gives
	\begin{equation*}
	Z(q):= \frac{q^{ -9 +\frac{{\rm tr}(M)}{2}}}{(2\pi i)^6} \text{PV} \int_{|w_j|=1} \frac{\left(w_1-w_1^{-1}\right)\left(w_2-w_2^{-1}\right)\left(w_5-w_5^{-1}\right)\left(w_6-w_6^{-1}\right)}{\left(w_3-w_3^{-1}\right)\left(w_4-w_4^{-1}\right)} \Theta_M(q;\b w)
	\prod_{j=1}^6 \frac{dw_j}{w_j},
	\end{equation*}
	where by \eqref{E:Theta2} (because $M$ is unimodular) we have
	\begin{equation}	
	\Theta_M(q;\b{w}) =\sum_{\b{m}\in  \mathbb Z^6}q^{\frac12 \b m^T M^{-1}\b m}e^{2\pi i\b m^T\b z}.
	\end{equation}
Applying \eqref{int1} and \eqref{int2} we find that
\begin{equation*}
Z(q) = \frac{q^{ -9 +\frac{{\rm tr}(M)}{2}}}{4} \sum_{\b{r} = \left(\varepsilon_1, \varepsilon_2, 2n_1, 2n_2, \varepsilon_5, \varepsilon_6\right)^T \atop \varepsilon_j \in \{ \pm 1 \}, (n_1,n_2) \in \mathbb{Z}^2} (\varepsilon_1 \varepsilon_2 \varepsilon_5 \varepsilon_6) {\rm sgn}^*(n_1) {\rm sgn}^*(n_2) q^{\frac12 \b{r}^T M^{-1} \b{r}}.
\end{equation*}
Applying Lemma \ref{L:M=A+c} completes the proof.	
\end{proof}

\noindent We are now ready to prove Theorem \ref{T:mainquantum}.
\begin{proof}[Proof\nopunct]{\em of Theorem \ref{T:mainquantum}.}
By splitting the summation over $\mathbb{Z}^2$ in (\ref{qqq}) into summations over $\N_0 \times \N_0$,  $(-\N) \times (-\N)$,
$\N_0 \times (-\N)$, and  $(-\N) \times \N_0$, a case-by-case computation for each unimodular matrix (\ref{E:Mb}) gives
\begin{equation*}
\sum_{\b \alpha \in \mathcal S } (-1)^{j+1} \sum_{\b m\in \mathbb{Z}^2} {\rm sgn}^*(m_1) {\rm sgn}^*(m_2) q^{Q(\b m +\b \alpha)}= \mathcal{Z}_1(q)-\mathcal{Z}_2(q)
\end{equation*}
where
$$\mathcal{Z}_1(q):=\mathcal{Z}_{Q, \b r, \b s}(q)=\sum_{j \in \{1,2\}} (-1)^{j+1} \sum_{\b \alpha \in \mathcal{S}_j}\sum_{\b n \in \N_0^2} q^{L Q(\b n +\b \alpha)}$$
$$\mathcal{Z}_2(q):=\mathcal{Z}_{Q^*, \b r, \b s}(q)=\sum_{j \in \{1,2 \}} (-1)^{j+1} \sum_{\b \alpha \in \mathcal{S}_j}\sum_{\b n \in \N_0^2} q^{L Q^*(\b n +\b \alpha)}$$
and $Q^*(\b n):=Q(-n_1,n_2)$.

The quadratic form $Q$ and constants $N_1,N_2,r_1,r_2,s_1,s_2$ (recall, $L={\rm gcd}(N_1,N_2)$) are given in the appendix.
In Section 4, Theorem \ref{main-thm} establishes that $\mathcal{Z}_1(q)$ is a quantum modular form of weight one and depth two on $\mathbb{Q}$.
The same result also applies to $\mathcal{Z}_2(q)$. Finally, we let $c_M:=9-\frac{1}{2}{\rm tr}(M)- c$, where $c$ is also listed in the appendix.
\end{proof}

\setcounter{secnumdepth}{0}
\section{Appendix: Data for positive unimodular matrices}
\label{A:uni}

Here we list all positive unimodular matrices of the form \eqref{E:Mb}, and the corresponding parameters that appear in $Z(q)$ (see \eqref{E:S1S2def} and Proposition \ref{prop:Zq}). In each case one can directly check that the assumptions in Section \ref{sec:family} are satisfied.

The value of $c$ and the quadratic form $Q$ are given below, and the data for $\mathcal{S}_j$ are presented in condensed form.

{\small

\noindent{\bf 1.} $M(2,3,7,1,2,3)$

\smallskip

\noindent  $Q(\b n)= n_1^2 + 12 n_1 n_2  + 37 n_2^2$, $c=\frac56$, $N_1=N_2=12$, $r_1=r_2=1$, $s_1=s_2=5$.

\vskip 3mm

\noindent{\bf 2.} $M(2,7,4,1,5,2)$

\smallskip

\noindent
$Q(\b n)=21 n_1^2 + 140 n_1 n_2 + 235 n_2^2$, $c=\frac{47}{70}$, $N_1=28$, $N_2=20$, $r_1=5$, $s_1=9$, $r_2=3$, $s_2=7$.

\vskip 3mm

\noindent {\bf 3.}  $M(6,31,3,1,2,7)$

\smallskip

\noindent  $Q(\b n)=465 n_1^2 + 2604 n_1 n_2 + 3647 n_2^2$, $c=\frac{274}{651}$, $N_1=372$, $N_2=28$, $r_1=149$, $s_1=161$, $r_2=5$, $s_2=23$.

\vskip 3mm

\noindent {\bf 4.}  $M(7,18,3,1,2,7)$

\smallskip

\noindent
$Q(\b n)=45 n_1^2 + 252 n_1 n_2 + 353 n_2^2$, $c=\tfrac{53}{126}$, $N_1=252$, $N_2=28$, $r_1=101$, $s_1=115$, $r_2=5$, $s_2=9$.

\vskip 3mm

\noindent {\bf 5.} $M(3,11,3,1,2,9)$

\smallskip

\noindent  $Q(\b n)=77 n_1^2 + 396 n_1 n_2 + 510 n_2^2$, $c=\tfrac{205}{396}$, $N_1=66$, $N_2=36$, $r_1=19$, $s_1=25$, $r_2=7$, $s_2=11$.
\vskip 3mm

\noindent {\bf 6.} $M(2,19,3,1,2,11)$

\smallskip

\noindent
$Q(\b n)=171 n_1^2 + 836 n_1 n_2 + 1023 n_2^2$, $c=\frac{239}{418}$, $N_1=76$, $N_2=44$, $r_1=17$, $s_1=21$, $r_2=9$, $s_2=13$.

\vskip 3mm

\noindent {\bf 7.} $M(2,3,3,1,2,27)$

\smallskip

\noindent  $Q(\b n)=25 n_1^2 + 108 n_1 n_2 + 117 n_2^2$, $c=\frac{37}{54}$, $N_1=12$, $N_2=108$, $r_1=1$, $s_1=5$, $r_2=25$, $s_2=29$.

\vskip 3mm

\noindent {\bf 8.}  $M(2,3,3,1,3,5)$

\smallskip

\noindent  $Q(\b n)=14 n_1^2  + 60 n_1 n_2 + 65 n_2^2$, $c=\frac{41}{60}$, $N_1=12$, $N_2=30$, $r_1=1$, $s_1=5$, $r_2=7$, $s_2=13$.

\vskip 3mm

\noindent {\bf 9.} $M(2,11,3,1,3,4)$

\smallskip

\noindent  $Q(\b n)=55 n_1^2 + 264 n_1 n_2 + 318 n_2^2$, $c=\frac{155}{264}$, $N_1=44$, $N_2=24$, $r_1=9$, $s_1=13$, $r_2=5$, $s_2=11$.

\vskip 3mm

\noindent {\bf 10.} $M(3,4,3,1,3,4)$

\smallskip

\noindent
$Q(\b n)=5 n_1^2 +  24 n_1 n_2 + 29 n_2^2$, $c=\tfrac{155}{264}$, $N_1=N_2=24$, $r_1=5$, $s_1=11$, $r_2=5$, $s_2=11$.

\vskip 3mm

\noindent {\bf 11.} $M(3,7,2,1,3,97)$

\smallskip

\noindent  $Q(\b n)=1337 n_1^2+4074 n_1 n_2+3104 n_2^2$, $c=\frac{835}{2037}$, $N_1=42$, $N_2=582$, $r_1=11$, $s_1=17$, $r_2=191$, $s_2=197$.

\vskip 3mm

\noindent {\bf 12.} $M(3,8,2,1,3,56)$

\smallskip

\noindent  $Q(\b n)=109 n_1^2+ 336 n_1 n_2+ 259 n_2^2$, $c=\tfrac{17}{42}$, $N_1=48$, $N_2=336$, $r_1=13$, $s_1=19$, $r_2=109$, $s_2=115$.

\vskip 3mm

\noindent {\bf 13.} $M(3,47,2,1,3,17)$

\smallskip

\noindent
 $Q(\b n)=1457 n_1^2 + 4794 n_1 n_2 + 3944 n_2^2$, $c=\frac{895}{2397}$, $N_1=282$, $N_2=102$, $r_1=91$, $s_1=97$, $r_2=31$, $s_2=37$.

\vskip 3mm

\noindent {\bf 14.} $M(3,88,2,1,3,16)$

\smallskip

\noindent  $Q(\b n)=319 n_1^2 + 1056 n_1 n_2 + 874 n_2^2$, $c=\frac{391}{1056}$, $N_1=528$, $N_2=96$, $r_1=173$, $s_1=179$, $r_2=29$, $s_2=35$.

\vskip 3mm

\noindent {\bf 15.}$M(4,5,2,1,3,47)$

\smallskip

\noindent
$Q(\b n)=1820 n_1^2 + 5640 n_1 n_2 + 4371 n_2^2$, $c=\frac{2263}{5640}$, $N_1=40$, $N_2=282$, $r_1=11$, $s_1=19$, $r_2=91$, $s_2=97$.

\vskip 3mm

\noindent {\bf 16.} $M(4, 77, 2, 1, 3, 11)$

\smallskip

\noindent  $Q(\b n)=532 n_1^2 + 1848 n_1 n_2 + 1605 n_2^2 ,c=\frac{635}{1848}$, $N_1=616$, $N_2=66$, $r_1=227$, $s_1=235$, $r_2=19$, $s_2=25$.

\vskip 3mm

\noindent {\bf 17.} $M(5, 16, 2, 1, 3, 11)$

\smallskip

\noindent  $Q(\b n)=1520 n_1^2 + 5280 n_1 n_2 + 4587 n_2^2,c=\frac{1813}{5280}$, $N_1=160$, $N_2=66$, $r_1=59$, $s_1=69$, $r_2=19$, $s_2=25$.

\vskip 3mm

\noindent {\bf 18.}  $M(7, 92, 2, 1, 3, 8)$

\smallskip

\noindent   $Q(\b n)=2093 n_1^2 +7728 n_1 n_2 + 7134
 n_2^2,c=\frac{2365}{7728}$, $N_1=1288$, $N_2=48$, $r_1=545$, $s_1=559$, $r_2=13$, $s_2=19$.

\vskip 3mm

\noindent {\bf 19.}  $M(8, 35, 2, 1, 3, 8)$

\smallskip

\noindent  $Q(\b n)=455 n_1^2 + 1680 n_1 n_2 + 1551  n_2^2),c=\frac{257}{840}$, $N_1=560$, $N_2=48$, $r_1=237$, $s_1=253$, $r_2=13$, $s_2=19$.

\vskip 3mm

\noindent {\bf 20.}  $M(11, 16, 2, 1, 3, 8)$

\smallskip

\noindent  $Q(\b n)=286 n_1^2 + 1056 n_1 n_2 + 975 n_2^2,c=\frac{323}{1056}$, $N_1=352$, $N_2=48$, $r_1=149$, $s_1=171$, $r_2=13$, $s_2=19$.

\vskip 3mm

\noindent {\bf 21.}  $M(12, 133, 2, 1, 3, 7)$

\smallskip

\noindent  $Q(\b n)=836 n_1^2 + 3192 n_1 n_2 + 3047 n_2^2,c=\frac{905}{3192}$, $N_1=3192$, $N_2=42$, $r_1=1451$, $s_1=1475$, $r_2=11$, $s_2=17$.

\vskip 3mm

\noindent {\bf 22.} $M(13, 72, 2, 1, 3, 7)$

\smallskip

\noindent $Q(\b n)=3432 n_1^2 + 13104 n_1 n_2 + 12509n_2^2,c=\frac{3715}{13104},N_1=1872,N_2=42,r_1=851, s_1=877, r_2=11, s_2=17$.

\vskip 3mm

\noindent {\bf 23.} $M(3, 4, 2, 1, 4, 23)$

\smallskip

\noindent  $Q(\b n)=195 n_1^2 + 552 n_1 n_2 + 391 n_2^2$, $c=\frac{121}{276}$, $N_1=24$, $N_2=184$, $r_1=5$, $s_1=11$, $r_2=65$, $s_2=73$.

\vskip 3mm

\noindent {\bf 24.} $M(3, 10, 2, 1, 4, 9)$

\smallskip

\noindent  $Q(\b n)=115 n_1^2 + 360 n_1 n_2 + 282 n_2^2,c=\frac{143}{360}$, $N_1=60$, $N_2=72$, $r_1=17$, $s_1=23$, $r_2=23$, $s_2=31$.

\vskip 3mm

\noindent {\bf 25.}  $M(3, 52, 2, 1, 4, 7)$

\smallskip

\noindent  $Q(\b n)=663 n_1^2 + 2184 n_1 n_2 + 1799 n_2^2 ,c=\frac{407}{1092}$, $N_1=312$, $N_2=56$, $r_1=101$, $s_1=107$, $r_2=17$, $s_2=25$.

\vskip 3mm

\noindent {\bf 26.}  $M(6, 67, 2, 1, 4, 5)$

\smallskip

\noindent  $
Q_1(\b{n}) = 2211n_1^2 + 8040n_1 n_2 + 7310n_2^2$, $c=\tfrac{2539}{8040}$, $N_1=804$, $N_2=40$, $r_1=329$, $s_1=341$, $r_2=11$, $s_2=19$.

\vskip 3mm

\noindent {\bf 27.} $M(2, 7, 2, 1, 4, 77)$

\smallskip
\noindent  $Q(\b n)=227 n_1^2 +616 n_1 n_2 + 418 n_2^2),c=\frac{279}{616}$, $N_1=28$, $N_2=616$, $r_1=5$, $s_1=9$, $r_2=227$, $s_2=235$.

\vskip 3mm

\noindent {\bf 28.}  $M(7, 26, 2, 1, 4, 5)$

\vskip 3mm

\noindent  $Q(\b n)=1001 n_1^2 + 3640 n_1 n_2 + 3310 n_2^2 ,c=\frac{1149}{3640}$, $N_1=364$, $N_2=40$, $r_1=149$, $s_1=163$, $r_2=11$, $s_2=19$.

\vskip 3mm

\noindent {\bf 29.} $M(2, 11, 2, 1, 4, 25)$

\smallskip

\noindent  $Q(\b n)=781 n_1^2 + 2200 n_1 n_2 + 1550  n_2^2 ,c=\frac{969}{2200}$, $N_1=44$, $N_2=200$, $r_1=9$, $s_1=13$, $r_2=71$, $s_2=79$.
\vskip 3mm

\noindent {\bf 30.}  $M(2, 19, 2, 1, 4, 17)$
\smallskip

\noindent  $Q(\b n)=893 n_1^2 + 2584 n_1 n_2 + 1870 n_2^2 ,c=\frac{1113}{2584}$, $N_1=76$, $N_2=136$, $r_1=17$, $s_1=21$, $r_2=47$, $s_2=55$.

\vskip 3mm

\noindent {\bf 31.} $M(2, 71, 2, 1, 4, 13)$

\vskip 3mm
\noindent  $Q(\b n)=2485 n_1^2 + 7384 n_1 n_2 + 5486 n_2^2 ,c=\frac{3105}{7384}$, $N_1=284$, $N_2=104$, $r_1=69$, $s_1=73$, $r_2=35$, $s_2=43$.
\vskip 3mm

\noindent {\bf 32.}  $M(3, 7, 2, 1, 5, 7)$

\smallskip

\noindent  $Q(\b n)=69 n_1^2 + 210 n_1 n_2 + 160 n_2^2 ,c=\frac{43}{105}$, $N_1=42$, $N_2=70$, $r_1=11$, $s_1=17$, $r_2=23$, $s_2=33$.

\vskip 3mm

\noindent {\bf 33.} $M(2, 5, 2, 1, 5, 33)$

\smallskip

\noindent  $Q(\b n)=254 n_1^2 + 660 n_1 n_2 + 429 n_2^2,c=\frac{307}{660}$,  $N_1=20$, $N_2=330$, $r_1=3$, $s_1=7$, $r_2=127$, $s_2=137$.

\vskip 3mm

\noindent {\bf 34.}  $M(2, 7, 2, 1, 5, 16)$

\smallskip

\noindent  $Q(\b n)=413 n_1^2 + 1120 n_1 n_2 + 760 n_2^2 ,c=\frac{507}{1120}$,  $N_1=28$, $N_2=160$, $r_1=5$, $s_1=9$, $r_2=59$, $s_2=69$.

\vskip 3mm

\noindent {\bf 35.}$M(2, 21, 2, 1, 5, 9)$

\smallskip

\noindent  $Q(\b n)=434 n_1^2 + 1260 n_1 n_2 + 915 n_2^2 ,c=\frac{541}{1260}$,  $N_1=84$, $N_2=90$, $r_1=19$, $s_1=23$, $r_2=31$, $s_2=41$.

\vskip 3mm

\noindent {\bf 36.}  $M(2, 55, 2, 1, 5, 8)$

\smallskip

\noindent  $Q(\b n)=297 n_1^2 + 880 n_1 n_2 + 652  n_2^2 ,c=\frac{371}{880}$,  $N_1=220$, $N_2=80$, $r_1=53$, $s_1=57$, $r_2=27$, $s_2=37$.

\vskip 3mm

\noindent {\bf 37.} $M(2, 3, 2, 1, 8, 57)$

\smallskip

\noindent  $Q(\b n)= 391 n_1^2 + 912 n_1 n_2 + 532 n_2^2 ,c=\frac{445}{912}$, $N_1=12$, $N_2=912$, $r_1=1$, $s_1=5$, $r_2=391$, $s_2=407$.

\vskip 3mm

\noindent {\bf 38.} $M(2, 3, 2, 1, 9, 32)$

\smallskip

\noindent  $Q(\b n)= 247 n_1^2 + 576 n_1 n_2 + 336 n_2^2 ,c=\frac{281}{576}$,  $N_1=12$, $N_2=576$, $r_1=1$, $s_1=5$, $r_2=247$, $s_2=265$.
\vskip 3mm

\noindent {\bf 39.}  $M(2, 3, 2, 1, 12, 17)$

\smallskip

\noindent  $Q(\b n)=175 n_1^2 + 408 n_1 n_2 + 238 n_2^2 ,c=\frac{199}{408}$,  $N_1=12$, $N_2=408$, $r_1=1$, $s_1=5$, $r_2=175$, $s_2=199$.

}

\end{document}